\documentclass{article}
\usepackage[a4paper, margin=1in]{geometry}
\usepackage[utf8]{inputenc}
\usepackage{hyperref}
\usepackage{enumitem}
\usepackage{xcolor}
\usepackage{ulem}
\usepackage{amsthm}
\usepackage{amsmath}
\usepackage[numbers]{natbib}
\usepackage{comment}

\usepackage{
    mathtools,
    amsfonts,
    amssymb,
    amsthm,
    nicefrac,
    empheq
}

\newtheorem{thm}{Theorem}

\newtheorem{prop}{Proposition}
\newtheorem{remark}{Remark}

\newtheorem{lem}[thm]{Lemma}
\newtheorem{ass}{Assumption}

\newtheorem*{notation}{Notation}

\renewenvironment{proof}[1][Proof]{%
    \par\noindent{\bfseries #1. }}{\qed}

\newlist{condlist}{enumerate}{1}
\setlist[condlist, 1]{label=(\Alph*)}

\renewcommand{\div}{\operatorname{div}}

\newcommand{\tr}{\text{Tr}}

\newcommand{\R}{\mathbb{R}}
\newcommand{\N}{\mathbb{N}}
\newcommand{\Z}{\mathbb{Z}}
\newcommand{\T}{\mathbb{T}}
\newcommand{\HH}{\mathbb{H}}
\newcommand{\E}{\mathbb{E}}
\newcommand{\lie}{\text{£}}
\newcommand{\LL}{\mathbb{L}}
\newcommand{\prob}{\mathbb{P}}
\newcommand{\curl}{\text{curl}}
\newcommand{\W}{\mathbb{W}}

\newcommand{\norm}[1]{\left\lVert #1 \right\rVert}

\newcommand{\indicator}{\mathbb{I}}

 \title{A uniform particle approximation to the Navier-Stokes-alpha models in three dimensions with advection noise}
\author{Filippo Giovagnini\footnote{Department of Mathematics, Imperial College London}, Dan Crisan\footnote{Department of Mathematics, Imperial College London}}

\begin{document}
\maketitle

\begin{abstract}
In this work, we investigate the properties of a system of interacting particles governed by a set of stochastic differential equations. The goal of the investigation  is to prove that the empirical measure associated with the system of particles converges uniformly—both in time and space—to the solution of the \textbf{three-dimensional Navier--Stokes-alpha model with advection noise}. This convergence provides a probabilistic framework for deriving macroscopic stochastic fluid equations from underlying microscopic dynamics. The analysis meakes use of semigroup techniques to address the nonlinear structure of the limiting equations. We provide a detailed treatment of the well-posedness of the limiting stochastic partial differential equation which helps to ensure that the particle approximation remains stable and controlled over time. Although similar convergence results have been obtained in two-dimensional settings, our study presents the first proof of strong uniform convergence for a stochastic fluid model in three dimensions  derived from an interacting particle system. We note that our results also cover the deterministic regime (with the advection noise being absent)—where this type of convergence had not been previously established either.
\end{abstract}

\section{Introduction}
\label{sec:intro}

The Navier–Stokes–$\alpha$ (NS-$\alpha$) model, also known as the Lagrangian-averaged Navier–Stokes model, is a regularized formulation of the classical Navier–Stokes equations designed to model the macroscopic behaviour of turbulent fluid flows while attenuating small-scale fluctuations. This is achieved through the introduction of a characteristic length scale parameter $\alpha$, which induces a smoothing of the velocity field. The resulting system exhibits improved numerical stability and reduced computational complexity, making it well-suited for large-scale simulations. Originally developed in the context of turbulence modelling and geophysical applications, the NS–$\alpha$ model retains essential physical features of the underlying fluid system, including key conservation laws and an appropriate energy balance. It thus provides a mathematically tractable yet physically consistent framework for the study of complex flow phenomena.


The NS--$\alpha$ model, originally introduced in~\cite{foias}, provides a regularized framework for fluid dynamics by introducing a mechanism that penalizes the formation of small-scale structures beneath a characteristic length scale $\alpha$. This penalization results in a nonlinearly dispersive modification to the classical Navier--Stokes equations. In contrast to viscous dissipation, which removes energy from the system, the $\alpha$-modification alters the nonlinear advection term in a scale-dependent manner, effectively inhibiting the transfer of energy to finer spatial scales. The model is also commonly referred to as the viscous Camassa--Holm equations, as discussed in~\cite{MR1745983, MR1719962, MR1721139, MR1878243}.

In ~\cite{foias}, the authors derive the NS--$\alpha$ equations via a Lagrangian-averaged interpretation of Kelvin’s circulation theorem, illustrating how the $\alpha$-modification naturally emerges in this context. A key analytical result is presented in Section 3: in three dimensions, the translational kinetic energy spectrum of the NS--$\alpha$ model exhibits a steeper decay of $k^{-3}$ beyond the cutoff wavenumber associated with $\alpha$, in contrast to the classical Kolmogorov $k^{-5/3}$ spectrum first introduced in \cite{kolmogorov}. This enhanced decay reduces the extent of the inertial range, thereby improving the model’s computational efficiency while retaining key features of turbulent flow dynamics.

Now, the so-called \emph{point vortex method} is a classical and computationally efficient technique used to model two-dimensional incompressible, inviscid fluid flows, particularly in the context of vortex-dominated dynamics. In this approach, the vorticity field is discretized into a finite set of idealized, singular vortices, each represented as a point carrying a prescribed circulation. The evolution of the fluid is then governed by the mutual interactions of these point vortices, as dictated by the Biot–Savart law. Originating from Helmholtz’s and Kirchhoff’s foundational work in the 19th century, the method captures key features of vortex dynamics, such as advection and vortex pairing, and has been widely employed in theoretical and numerical studies of vortex sheets, wake flows, and turbulent structures. 

The theoretical foundations of the point vortex method trace back to the pioneering work of Hermann von Helmholtz and Gustav Kirchhoff in the 19th century. In \cite{helmholtz1858} the author formulated the fundamental laws governing vortex motion in ideal fluids, introducing the concept of vortex lines and demonstrating their conservation under inviscid, incompressible flow. Building on this, in \cite{kirchhoff1876} the author provided a mathematical framework for the dynamics of point vortices in two dimensions, showing that their motion can be derived from a Hamiltonian formulation. This elegant reduction of fluid dynamics to a finite-dimensional system laid the groundwork for the modern point vortex method and remains central to the analytical study of vortex dynamics.

The origins of the {\it modern} point-vortex method can be traced to the work of~\cite{rosenhead}, while the development of computational techniques that make use vortex dynamics for simulating viscous fluid flows began with the seminal contributions of~\cite{chorin_main} (see also~\cite{Chorin1994VorticityAT, chorin1973discretization}). Chorin's work introduced several fundamental ideas that helped establish point-vortex approximations as practical tools in numerical fluid dynamics. Although these approximations do not exactly solve the governing equations of fluid motion, and their convergence to true solutions remains analytically subtle, significant progress has been achieved in proving convergence for a variety of vortex-based numerical schemes applied to the Navier–Stokes (NS) and to other related fluid equations.

The convergence of the random vortex method for two-dimensional incompressible viscous fluid flows was studied in \cite{goodman1987convergence}, with improvements to the rate of convergence made in \cite{Long1988}. Both studies used the distinctive property of two-dimensional vortex motions: the transport of vortices without nonlinear stretching.

The study of three-dimensional inviscid fluid flows based on vortex dynamics has been a prominent area of research. Reviews of progress in this field can be found in \cite{madja_bertozzi} and \cite{cottet_koumoutsakos}. The first three-dimensional vortex methods for inviscid fluid flows were introduced by \cite{anderson_greengard}, followed by a proof of their convergence. The convergence of three-dimensional vortex methods in a Lagrangian framework was established in \cite{beale_madja_I} and \cite{beale_majda_II}, while a new grid-free vortex method was later proposed in \cite{beale_grid_free}. In \cite{cottet_goodman_hou}, a vortex method without mollifying the singular Biot-Savart kernel was introduced. For viscous flows, numerical methods based on vortex dynamics for solving fluid dynamics equations were studied in works such as \cite{knio_ghoniem}. In \cite{qian}, the authors derive exact random vortex dynamics equations for the three-dimensional incompressible NS equations, which form the basis of this work.

Several other stochastic representations of solutions to the three-dimensional Navier-Stokes equations exist in the literature. For example, in \cite{constantinStochasticLagrangianRepresentation2008}, a stochastic Lagrangian representation for the velocity of an incompressible fluid flow is obtained using stochastic flows and Leray’s projection. The existence and uniqueness of solutions to the Navier-Stokes equation via stochastic flow methods were studied in \cite{zhang}, while the connection between mild solutions and stochastic representations was explored in \cite{olivera_representation}. Alternative approaches proposed in \cite{brusnello_flandoli_romito} and \cite{albeverio_belopolskaya} led to McKean-Vlasov-type stochastic differential equations (SDEs). A propagation-of-chaos-type result for the three-dimensional Navier-Stokes equations was conjectured in \cite{esposito_pulvirenti} and later proven in \cite{fontbona}.

In this work, we study a smooth, mollified version of the empirical measure associated with a particle system and establish strong uniform convergence to the three-dimensional Navier--Stokes--$\alpha$ model with advection noise (see~\cite{paper} and~\cite{giovagnini2024uniformpointvortexapproximation} for the two-dimensional case). In the absence of stochastic perturbations, our result recovers the corresponding convergence for the deterministic three-dimensional Navier--Stokes--$\alpha$ model. To the best of our knowledge, this type of strong uniform mean-field convergence has not been previously established for {\it any} three-dimensional fluid equations.

The paper is structured as follows. In Section \ref{sec:particle_system} we provide a rigorous definition of the particle system and state the main result of the paper. In Section \ref{sec:compactness} we establish the compactness of the laws associated with the sequence of approximating particle systems. In Section \ref{sec:passing-to-the-limit} we prove that the limit of the particle system satisfies the three-dimensional stochastic Navier-Stokes equation. Section \ref{sec:well-posed-limit} addresses the well-posedness of the limiting equation. Finally, Section \ref{sec:conclusions} contains the convergence in probability of the original sequence, ensuring that the result holds without the need for subsequence arguments.

We conclude the introduction with a precise formulation of the stochastic NS--$\alpha$ model, along with a list of standard notation that will be used throughout the paper. In particular, we consider the following stochastic partial differential equation:  

\begin{equation}
\label{eq:NS_alpha_model}
\begin{cases}
d\omega(t, x)+ \lie_{u} \omega(t, x) dt + \sum\limits_{k \in \N} \lie_{\sigma_k} \omega(t, x) \circ dW^k_t= \nu \Delta \omega(t, x), \quad x \in \T^3,\\
u = (\indicator - \alpha^2 \Delta)^{-1} K \star \omega, \\
\omega(0, \cdot) = \omega_0.
\end{cases}
\end{equation}
Here, $\T^3:=\R^3/\Z^3$ is the $3$-dimensional torus,  $u: \T^3 \to \R^3$ denotes the velocity field of a fluid, $\omega := \curl(u)$ is its vorticity, and $K$ is the Biot-Savart kernel in the $3$-torus (see equation \eqref{eq:biot_savart} in the appendix for an explicit definition). In equation \eqref{eq:NS_alpha_model}, $\lie_{\sigma} \omega$ is the Lie derivative of $\omega$ along the vector field $\sigma$ as:
$
\lie_{\sigma} \omega := \left(\nabla \omega \right) \sigma - \left(\nabla \sigma \right)^T \omega,
$
and $W^k$ is a sequence of independent $\R^3$-valued Brownian Motion defined in a filtered probability space $\left( \Omega, \mathcal{F}, \mathbb{P},\left( \mathcal{F}_t \right)_{t \geq 0} \right)$. In this paper we will work with solutions to equation \eqref{eq:NS_alpha_model} with values in the space $\LL^2\left([0, T]; \HH^{-1}_2(\T^3) \right)$. See below for the definition oof the Sobolev space  $\HH^{-1}_2(\T^3)$.

In this paper, we introduce a representation
of the fluid vorticity which is the solution of equation
(\ref{eq:NS_alpha_model}, as the limit a weighted empirical distribution of system of "point" vorticities. In other words,  we deduce and rigorously justify the so-called point vortex method for the stochastic Navier--Stokes--$\alpha$ (NS--$\alpha$) model. 

We remark that in two-dimensional models (see, e.g., \cite{marchioro_pulvirenti}), the point vortex method represents the vorticity as a collection of vortices with equal scalar weights, reflecting the scalar nature of vorticity in 2D flows. This simplification, however, does not extend to the three-dimensional setting. In 3D, vorticity is a vector-valued field rather than a scalar, and as a consequence, point vortices must carry vector-valued weights.     

\begin{notation}
$\left.\right.$\\[-5mm]
\begin{itemize}
    \item $\text{Id} \in \R^{3 \times 3}$ denotes the identity matrix.
    \item $\wedge : \R^3 \times \R^3 \to \R^3$ is the usual cross product in $\R^3$.
    \item If $f, g:\T^3 \to \R^3$, we define $f \star g : \T^3 \to \R^3$ as follows:
    \[
    f \star g (x) := \int_{\T^3} f(x-y) \wedge g(y) dy, \quad x \in \T^3.
    \]
    \item If $p \geq 1$, $\LL^p(\T^3; \R^3)$ denotes the usual $\LL^p$ space on the torus.
    \item For $p \geq 1$ and $\alpha \in \R$:
    \[
    \HH^{\alpha}_p(\T^3):= \left\{ f \in \LL^p(\T^3) \bigg| \sum\limits_{k \in \Z^3 } \left(1 + |k|^2 \right)^{\alpha/2} c_k(f) \textbf{e}_k(x) \in \LL^p(\T^3)\right\},
    \]
    where $\textbf{e}_k: \T^3 \to \mathbb{C}$ is defined as:
    \[
    \textbf{e}_k(x) := \frac{1}{(2 \cdot \pi)^{\frac{3}{2}}} e^{i <k, x>}.
    \]
    and where $c_k(f)$ is the $k$-th Fourier coefficient of $f$. For a connection with the usual definition with the classical Sobolev spaces, see \cite{hitchhiker}.
    \item For a $C^2$ function $f:\T^3 \to \R$, $Hf(x)$ denotes the Hessian matrix of $f$.
    \item $\Delta : \HH^{2}_p(\T^3) \subset \LL^p(\T^3) \to \LL^p(\T^3)$ is the usual Laplacian.
    \item $\indicator : \LL^p(\T^3) \to \LL^p(\T^3)$ is the identity operator.
    \item $(\indicator - \alpha^2 \Delta)^{-1} :  \LL^p(\T^3) \to \HH^{2}_p(\T^3)$ is the inverse of the differential operator $\indicator - \alpha^2 \Delta$.
    \item For $-\alpha > 0$ the following characterization holds:
    \[
    \HH^{-\alpha}_p(\T^3) = \left( \HH^{\alpha}_p(\T^3) \right)^*
    \]
    where $\left( \cdot \right)^*$ denotes the topological dual of a Banach space.
    \item $\HH^{\alpha}_p(\T^3; \R^3)$ is defined in the same way as $\HH^{\alpha}_p(\T^3)$, but component wise.
    \item $C([0, T]; \HH^{\alpha}_p(\T^3; \R^3))$ is the space of functions $f:[0, T] \to \HH^{\alpha}_p(\T^3; \R^3) $ that are continuous on $[0,T]$.
    \item For two functions $f,g:D \to \mathbb{\R^n}$ we will denote:
    \[
    \left\langle f,g \right\rangle:= \int_D f(x) \cdot g(x) dx,
    \]
    where $D=\T^3$ or $D=\R^3$.
\end{itemize}
\end{notation}

\section{The particle system}
\label{sec:particle_system}

In the following, we will choose constants $p,\alpha,$ and $\beta$ such that $p>6$,  $6/p < \alpha < 1$, $\eta \in (6/p, \alpha)$
and $0 < \beta < \frac{1}{3 + \alpha - 6/p}$. 
\footnote{We choose these constants to ensure 
the compact embeddings:
$\HH^{\alpha}_p(\T^3) \subset \HH^{\eta}_p(\T^3) \subset \HH^{-2}_2(\T^3)$.
See \cite{hitchhiker} for the proof of this result. The constrain on the constant $\beta$ is to ensure the convergence of the mollified empirical distribution of the vortex particles.}
Let $V:\mathbb{R}\rightarrow  \mathbb{R}$ be a smooth function which we use to construct a sequence of mollifiers $V^N:\mathbb{R}\rightarrow  \mathbb{R}$ such that $V^N(x):=N^{3\beta}V(N^{\beta}x)$. An example of a funcion $V$ is the following:
\begin{equation}
    V(x):=\begin{cases}
        c \cdot e^{\frac{1}{\pi^2-4|x|^2}} \text{ if } |x| \leq \frac{\pi}{2}, \\
        0 \text{ otherwise },
    \end{cases}
\end{equation}
where the constant $c$ is such that $\int_{\T^2} V(y)dy = 1$.
The following has been inspired by the work of \cite{fontbona}, and by the representation formula in section \ref{repform}. 
We will consider a system of particles whose trajectories are modelled by a set of stochastic progresses $\{ X^{i, N} \}_{i=1, \dots, N}$. Associated to each particle we consider a set of matrix-valued progresses $\{\Phi^{i, N}\}_{i=1, \dots, N}$. The progresses $\{X^{i, N}, \Phi^{i, N}\}_{i=1, \dots, N}$ will satisfy the following system of coupled stochastic differential equations:
\begin{equation}
\label{eq:particle_system}
\begin{cases}
\begin{aligned}
dX_t^{i, N} &= \frac{1}{N} \sum_{j \neq i} V^N \star K_{\alpha}\left(X_{t}^{i, N}-X_{t}^{j, N}\right) \wedge \Phi_{t, 0}^{j, N} h_{0}\left(X_{0}^{j}\right) d t + \sqrt{2 v} dB_t^i \\
& + \sigma \left(X_{t}^{i, N}\right) \circ dW_t,\\
X_0^{i, N} & = X_0^i,\\
d\Phi_{t, s}^{i, N} & = -\frac{1}{N} \sum_{j \neq i} \left[\nabla V^N \star K_{\alpha}\left(X_s^{i, N}-X_s^{j, N}\right) \wedge \Phi_{t,s}^{j, N}h_0\left(X_0^j\right)\right] \Phi_{t,s}^{i, N} ds \\
& - \nabla \sigma \left(X_{s}^{i, N}\right) \Phi^{i, N}_{t,s} \circ dW_s, \\
\Phi^{i, N}_{t,t} &= \text{Id}.
\end{aligned}
\end{cases}
\end{equation}
In equation \eqref{eq:particle_system}, $X_t^{i, N}$ is a $\T^3$-valued stochastic process, $\Phi_{t,s}^{i, N}$ is a $\R^3 \times \R^3$-valued stochastic process. The choice of this system is inspired by \cite{qian}. From now on, will write $\Phi^{i, N}_{t}$ for indicating $\Phi^{i, N}_{t,0}$. We will impose the following:
\begin{ass}
$\left.\right.$\\[-5mm]
\label{ass:assprinci}
\begin{enumerate}
    \item We assume $\sigma:\T^3 \to \R^3$ is such that:
    \[
    \div\left(\sigma \sigma^T + (\nabla \sigma) (\nabla \sigma)^T\right) = 0.
    \]
    A concrete example of an advection noise that satisfies this condition has been introduced by \cite{kazantsev}, and it has been studied in \cite{kraichnan_kazantsev}, just to mention one of many works in this direction.
    \item There exists a function $h_0:\T^3 \to \T^3$ such that the initial condition of the system  \eqref{eq:particle_system}, i.e. the family 
    $(X_0^i)_{i \in \N} \in \T^3$ satisfies:
    \[\lim_{N\rightarrow \infty}
    \frac{1}{N} \sum\limits_{i = 1}^N V^N(\cdot-X_0^i) h_0(X_0^i) =\omega_0(\cdot),
    \]
    in $\HH^{-1}_p(\T^3)$, in probability.
    \item There exist $p$ and $\alpha$ as above such that for all $q > 0$:
    \[
    \sup\limits_{N \in \N} \E \left[ \left\|\frac{1}{N} \sum\limits_{i = 1}^N V^N(\cdot-X_0^i) h_0(X_0^i)  \right\|_{\HH^{\alpha}_p}^q\right] < \infty.
    \]
\end{enumerate}
\end{ass}

Following from equation (\ref{eq:NS_alpha_model}), let us define the following mollified Biot-Savart operator:
$$K_{\alpha}(\omega) := (\indicator - \alpha^2 \Delta)^{-1} K \star \omega.$$ See equation \eqref{eq:biot_savart} for a more explicit definition.

Define, for a function \(f \in C^2 \left( \mathbb{T}^{3}; \mathbb{R} \right) \), the
following:
\begin{equation}
\label{eq:action_mu_test}
\left\langle \mu^{N}_{t}, f \right\rangle:= \frac{1}{N}\sum\limits_{i = 1}^{N}f(X_t^{i, N})\Phi_t^{i, N} h_{0}(X_{0}^i).
\end{equation}

Applying It\^o's
formula to the product
\(f(X_{t}^{i, N}) \cdot \Phi_{t}^{i, N}h_0(X_{0}^{i})\) one gets:
\[
\begin{aligned}
f\left(X_{t}^{i, N}\right) \Phi_{t}^{i, N}h_0(X_{0}^{i}) &= f\left(X_{0}^{i}\right) h_0(X_{0}^{i})+ \int_{0}^{t} \nu \Delta f\left( X_{s}^{i, N}\right) \Phi_{s}^{i, N}h_0(X_{0}^i) ds \\
& - \int_{0}^{t} \nabla f\left(X_{s}^{i, N}\right)\frac{1}{N} \sum_{j \neq i} \bigg[ V^N \star K_{\alpha}\left(X_{s}^{i, N}-X_{s}^{j, N}\right) \wedge \Phi_s^{j, N} h_{0}\left(X_{0}^{j}\right)\bigg] \Phi_{s}^{i, N}h_0(X_{0}^i)ds \\
& + \int_{0}^{t} f\left(X_{s}^{i, N}\right) \frac{1}{N} \sum_{j \neq i}\left[\nabla V^N \star K_{\alpha}\left(X_s^{i, N}-X_s^{j, N}\right) \wedge \Phi_s^{j, N} h_0\left(X_0^j\right)\right] \Phi_s^{i, N} h_0(X_{0}^{i})ds \\
& - \int_0^t \nabla f\left(X_{s}^{i, N}\right) \Phi_{s}^{i, N}h_0(X_{0}^i) \cdot \sigma(X_{s}^{i, N})\circ dW_{s} + \int_0^t f\left(X_{s}^{i, N}\right) \Phi_{s}^{i, N}h_0(X_{0}^i) \nabla \sigma(X_{s}^{i, N})\circ dW_{s} \\
& + \sqrt{2\nu}\int_0^t \nabla f\left( X_{s}^{i, N}\right) \Phi_{s}^{i, N} h_{0}(X_{0}^i) dB_{s}^{i}.
\end{aligned}
\]

Now, from the equation above,
considering the sum over all \(i \in \{1, \dots, N\}\), one has:
\begin{equation}
\label{eq:mu_N}
\begin{aligned}
&\left\langle \mu^{N}_{t}, f \right\rangle = \left\langle \mu^{N}_{0}, f \right\rangle + \nu \int_{0}^{t}\left\langle \mu^{N}_{s}, \Delta f \right\rangle ds - \int_{0}^{t} \left\langle \mu^{N}_{s}, u^{N}_{s} \cdot \nabla f \right\rangle ds + \int_{0}^{t} \left\langle \mu^{N}_{s}, f \cdot \nabla u^{N}_{s} \right\rangle ds \\
& - \int_{0}^{t}\left\langle \mu^{N}_{s}, \sigma \cdot \nabla f \right\rangle \circ dW +  \int_{0}^{t}\left\langle \mu^{N}_{s}, f \cdot \nabla \sigma \right\rangle \circ dW = \frac{1}{N} \sum\limits_{i = 1}^{N}\int_{0}^{t} \nabla f\left(s, X_{s}^{i, N}\right) \Phi_{s}^{i, N} h_{0}(X_{0}^i) dB_{s}^{i},
\end{aligned}
\end{equation}
where:
\[
u^N_t(x) := \frac{1}{N} \sum_{j \neq i} \bigg[ V^N \star K_{\alpha}\left(x-X_{s}^{j, N}\right) \wedge \Phi_s^{j, N} h_{0}\left(X_{0}^{j}\right)\bigg].
\]
Now let us define:
\[
\varphi^{N}_{x}(y):=V^{N}(x-y),
\]
and let us consider a process $g^N_t : \T^3 \to \R^3$ as follows:
\[
g^{N}(x):=V^N \star \mu^N_t(x)=\langle \mu^{N}, \varphi_x^{N} \rangle.
\]
Also, let us define the following process $\mathfrak{F}_t^{i, N}(x): \T^3 \to \R^3$ to help the exposition:
\begin{equation}
\label{eq:definition_term_g}
\mathfrak{F}_t^{i, N}(x) := V^{N}(x-X_{t}^{i, N}) \Phi_{t}^{i, N} h_{0}(X_{0}^{i}).
\end{equation}
We remark that the process $g^N_t$ is $C^{\infty}$ as it is a finite sum of $C^{\infty}$ functions:
\[
g^N_t(x) = \frac{1}{N}\sum\limits_{i=1}^N V^{N}(x-X_{t}^{i, N}) \Phi_{t}^{i, N} h_{0}(X_{0}^{i}).
\]

From equation \eqref{eq:mu_N}, one obtains:
\[
\begin{aligned}
g^{N}_{t}(x) &= g^{N}_{0}(x) - \int_{0}^{t} \Delta g^{N}_{t}(x) ds - \int_{0}^{t} \left\langle \mu^{N}_{s}, u^{N}_{s} \cdot \nabla V^{N}(x - \cdot) \right\rangle ds + \int_{0}^{t} \left\langle \mu^{N}_{s}, V^{N}(x - \cdot) \cdot \nabla u^{N}_{s} \right\rangle ds \\
& - \int_{0}^{t}\left\langle \mu^{N}_{s}, \sigma \cdot \nabla V^{N}(x - \cdot) \right\rangle \circ dW + \int_{0}^{t}\left\langle \mu^{N}_{s}, V^{N}(x - \cdot) \cdot \nabla \sigma \right\rangle\circ dW \\
&= \frac{\sqrt{2\nu}}{N} \sum\limits_{i = 1}^{N}\int_{0}^{t} \nabla V^{N}(x-X_{s}^{i, N}) \Phi_{s}^{i, N} h_{0}(X_{0}^i) dB_{s}^{i},
\end{aligned}
\]
where:
\[
\begin{aligned}
u^{N}_{s}(x) &:= \left( V^N \star K_{\alpha} \right) \star \mu^N_t(x) = \frac{1}{N} \sum_{j \neq i} \bigg[ V^N \star K_{\alpha}\left(x-X_{s}^{j, N}\right) \wedge \Phi_s^{j, N}h_{0}\left(X_{0}^{j}\right)\bigg] \\
&= \frac{1}{N} \sum_{j \neq i} \bigg[ \int_{\T^3} V^N(x-X^{j, N}_s - y)K_{\alpha}(y) dy \wedge \Phi_s^{j, N}h_{0}\left(X_{0}^{j}\right)\bigg] \\
&= \int_{\T^3} K_{\alpha}(y) \wedge \frac{1}{N} \sum_{j \neq i} \bigg[ V^N(x-X^{j, N}_s - y) \Phi_s^{j, N}h_{0}\left(X_{0}^{j}\right) \bigg] dy\\
&= \int_{\T^3} K_{\alpha}(y) \wedge \left(V^N \star \mu^N \right)(x - y) dy \\
& = K_{\alpha} \star \left( V^N \star \mu^N_t\right)(x) = K_{\alpha} \star g^N_t(x).
\end{aligned}
\]

\section{Relative compactness of the sequence}
\label{sec:compactness} 

In this section, we prove that the sequence of the laws of $g^N_t$ is tight in the following space:
\[
C([0,T], \HH^{\eta}_p(\T^3;\R^3)),
\]
for a certain $\eta \in (0,1)$.
First of all, if one passes to the It\^{o} integral in the initial system:
\[
\begin{aligned}
g^{N}_{t}(x) &= g^{N}_{0}(x) + \int_{0}^{t} \left[ \nu \Delta g^{N}_{t}(x) + \frac{1}{2}\sigma^{T}Hg^{N}_{t}\sigma(x) + \frac{1}{2} \tr \left( \left( \nabla \sigma \right)^{T}Hg^{N}_{t}\left( \nabla \sigma \right)(x)\right)\right]ds \\
&+ \int_{0}^{t} \left\langle \mu^{N}_{s}, \left( \sigma \cdot \nabla \sigma \right) \cdot \nabla V^{N}(x - \cdot) \right\rangle ds + \int_{0}^{t} \left\langle \mu^{N}_{s}, \left( \left( \nabla \sigma \right) \cdot \nabla \left( \nabla \sigma \right) \right) \cdot \nabla V^{N}(x - \cdot) \right\rangle ds\\
&- \int_{0}^{t} \left\langle \mu^{N}_{s}, u^{N}_{s} \cdot \nabla V^{N}(x - \cdot) \right\rangle ds + \int_{0}^{t} \left\langle \mu^{N}_{s}, V^{N}(x - \cdot) \cdot \nabla u^{N}_{s} \right\rangle ds \\
& - \int_{0}^{t}\left\langle \mu^{N}_{s}, \sigma \cdot \nabla V^{N}(x - \cdot) \right\rangle dW_s + \int_{0}^{t}\left\langle \mu^{N}_{s},  V^{N}(x - \cdot) \cdot \nabla \sigma \right\rangle dW_s \\
& + \frac{\sqrt{2\nu}}{N} \sum\limits_{i = 1}^{N}\int_{0}^{t} \nabla V^{N}(x-X_{s}^{i, N}) \Phi_{s}^{i, N} h_{0}(X_{0}^i) dB_{s}^{i}
\end{aligned}
\]

Now, since $g^N_t \in C^{\infty}(\T^3)$, one can consider the mild form. See Theorem 6.5 and Theorem 6.7 in \cite{Da_Prato_Zabczyk_2014} to have that the mild form coincides with the strong form. We get:
\begin{equation}
\label{eq:mild_form}
\begin{aligned}
g^{N}_{t} &= e^{t A} g^{N}_{0} + \int_{0}^{t} e^{(t-s)A} \left\langle \mu^{N}_{s}, \left( \sigma \cdot \nabla \sigma \right) \cdot \nabla V^{N}(x - \cdot) \right\rangle ds + \int_0^t e^{(t-s)A}\left\langle \mu^{N}_{s}, \left( \left( \nabla \sigma \right) \cdot \nabla \left( \nabla \sigma \right) \right) \cdot \nabla V^{N}(x - \cdot) \right\rangle ds \\
&- \int_{0}^{t}e^{(t-s)A} \left\langle \mu^{N}_{s}, u^{N}_{s} \cdot \nabla V^{N}(x - \cdot) \right\rangle ds + \int_{0}^{t} e^{(t-s)A} \left\langle \mu^{N}_{s}, V^{N}(x - \cdot) \cdot \nabla u^{N}_{s} \right\rangle ds \\
& - \int_{0}^{t} e^{(t-s)A} \left\langle \mu^{N}_{s}, \sigma \cdot \nabla V^{N}(x - \cdot) \right\rangle dW + \int_{0}^{t} e^{(t-s)A} \left\langle \mu^{N}_{s}, V^{N}(x - \cdot) \nabla \sigma  \right\rangle dW \\
&+ \frac{\sqrt{2\nu}}{N} \sum\limits_{i = 1}^{N}\int_{0}^{t} e^{(t-s)A}\nabla V^{N}(x-X_{s}^{i, N}) \Phi_{s}^{i, N} h_{0}(X_{0}^i) dB_{s}^{i}
\end{aligned}
\end{equation}
where $A: \HH^{2}_p(\T^3) \subset \LL^p(\T^3) \to \LL^p(\T^3)$ is defined as:
\begin{equation}
\label{eq:definition_A}
\left\langle A, f \right\rangle := \nu \Delta f + \frac{1}{2} \tr \left( \sigma\sigma^T Hf \right) + \frac{1}{2} \tr \left( \left( \nabla \sigma \right)\left( \nabla \sigma \right)^{T}(Hf)\right).
\end{equation}
We notice that, as a consequence of Corollary 2.3 (section 3.2 of \cite{pazy}), if $\norm{\sigma}_{\W^{1, \infty}} < C \nu$ where $C$ is a universal constant, then $A$ is the infinitesimal generator of an analytical semigroup on $\LL^p(\T^3)$. We will denote $\left( e^{tA} \right)_{t \geq 0}$ for the semigroup associated to the operator $A$. See section 2.5 in \cite{pazy} for an overview of the results about analytical semigroups.

\begin{prop}
\label{thm:compact_1}
For all $q \geq 2$, there exists a $C>0$ such that for all $N \in \N$:
\[
\sup\limits_{t \in [0, T]}\E \left[ \left\|g_t^{N}\right\|_{\HH^{\alpha}_p}^q \right]^{1/q} \leq C
\]
\end{prop}

\begin{proof}
Just a consequence of Lemma \ref{lem:compactness_noises} and Lemma \ref{lem:compact_advection_term}. More precisely, consider the equation satisfied by $g^N$ in mild form, i.e. equation \eqref{eq:mild_form}. With Lemma \ref{lem:compactness_noises} one can estimate the stochastic integrals. All the remaining terms can be bounded with Lemma \ref{lem:compact_advection_term}. One ends up with:
\[
\E \left[ \left\|g_t^{N}\right\|_{\HH^{\alpha}_p}^q \right]^{1/q} \leq C + C\int_0^t
\E \left[ \left\|g_s^{N}\right\|_{\HH^{\alpha}_p}^q \right]^{1/q} ds.
\]
Therefore one concludes with Gronwall's Lemma.
\end{proof}

For a proof of the following two Lemmas, see the appendix.
\begin{lem}
\label{lem:compactness_noises}
Let $p > 6$ and $\alpha > 6/p$. Let $q \geq 2$. Then there exists $C$, independent of $N$ and $t \in [0, T]$, such that:
\[
\begin{aligned}
\mathbb{E}\left[\left\|\frac{1}{N} \sum_{i=1}^N \int_0^t(I-A)^{\frac{\alpha}{2}} e^{(t-s) A} \nabla\left(V^N\left(\cdot-X_s^{i, N}\right)\right) \Phi^{i, N}_s h_0(X_0)d B_s^{i}\right\|_{\mathbb{L}^p}^q\right] & \leq C, \\ \mathbb{E}\left[\left\|\int_0^t(I-A)^{\frac{\alpha}{2}} e^{(t-s) A}\left(\nabla V^N \star \left(\sigma_k \mu_s^{N}\right)\right) d W_s^k\right\|_{\mathbb{L}^p}^q\right] & \leq C .
\end{aligned}
\]
\end{lem}

\begin{lem}
\label{lem:compact_advection_term}
Let $q \geq 2$. Then there exists a constant $C$, independent of $N$ and $t \in [0, T]$, such that:
\begin{equation}
\label{eq:mild_term_1}
\begin{aligned}
& \E \left[ \left\| \int_0^t (\indicator - A)^{\frac{\alpha}{2}}e^{(t-s)A} \left( \nabla V^N \star \left( K_{\alpha} \star g_s^{N}  \right) \mu_s^N \right) ds \right\|_{\LL^p}^q \right]^{1/q} \leq C \int_0^t \E \left[ \left\|g_s^{N}\right\|_{\LL^p}^q \right]^{1/q} ds.
\end{aligned}
\end{equation}
and:
\begin{equation}
\label{eq:mild_term_2}
\begin{aligned}
& \E \left[ \left\| \int_0^t (\indicator - A)^{\frac{\alpha}{2}}e^{(t-s)A} \left( V^N \star \left( \nabla K_{\alpha} \star g_s^{N}  \right) \mu_s^N \right) ds \right\|_{\LL^p}^q \right]^{1/q} \leq C \int_0^t \E \left[ \left\|g_s^{N}\right\|_{\LL^p}^q \right]^{1/q} ds.
\end{aligned}
\end{equation}
\end{lem}

\begin{remark}
\label{rem:remark_boundedness}
We remark that the following two norms are equivalent:
\[
\left\| f \right\|_{\HH^{\alpha}_p} \sim \left\| (\indicator - A)^{\alpha/2}f \right\|_{\LL^p(\T^{3})}.
\]
For an detailed explanation of this see the appendix.
Therefore Proposition \ref{thm:compact_1} tells us:
\[
\E \bigg[ \bigg\| (\indicator -A)^{\frac{\alpha}{2}} g_t^{N} \bigg\|_{\LL^p(\T^3)}^q \bigg] \leq C, \quad \text{ for all } t \in [0,T],
\]
from which we can easily conclude that, choosing $q >2$:
\[
\begin{aligned}
\E \bigg[ \| g^{N} \|^2_{\LL^q([0,T], \HH^{\alpha}_p(\T^3))} \bigg] &=
\E \bigg[ \left( \int_0^t\| g^{N}_s \|^q_{\HH^{\alpha}_p(\T^3)} ds \right)^{2/q}\bigg] \\
&\leq \E \bigg[ \int_0^t \left( \| g^{N}_s \|^q_{\HH^{\alpha}_p(\T^3)} \right)^{2/q} ds \bigg] \\
&= \E \bigg[ \int_0^t \| g^{N}_s \|^2_{\HH^{\alpha}_p(\T^3)} ds \bigg] \\
&= \int_0^t \E \bigg[ \| g^{N}_s \|^2_{\HH^{\alpha}_p(\T^3)} \bigg] ds \\
&\leq T \sup\limits_{s \in [0, T]} \E \bigg[ \| g^{N}_s \|^2_{\HH^{\alpha}_p(\T^3)} \bigg] \leq C.
\end{aligned}
\]
\end{remark}

Let us also state:
\begin{prop}
\label{thm:compact_2}
Let $\gamma \in \left(0, \frac{1}{2}\right)$ and $q^{\prime} \geq 2$. There exists a positive constant $C$ such that, for any $N \in \mathbb{N}$, it holds that:
\[
\mathbb{E}\left[\int_0^T \int_0^T \frac{\left\|g_t^N-g_s^N\right\|_{\HH^{-2}_2(\T^3)}^{q^{\prime}}}{|t-s|^{1+q^{\prime} \gamma}} d s d t\right] \leq C.
\]
\end{prop}
This tells us, in particular, that:
\[
\E \bigg[ \| g^{N} \|^2_{\HH^{\gamma}_{q'}([0,T], \HH^{-2}_{2}(\T^3))} \bigg] \leq C.
\]

For $p > 6$, $\alpha > 6/p$ to have the compact embeddings:
\[
\HH^{\alpha}_p(\T^3) \subset \HH^{\eta}_p(\T^3) \subset \HH^{-2}_2(\T^3).
\]
See \cite{hitchhiker} for the proof the compact embeddings above.
Now, from Corollary 9 of \cite{simon}, for any \(\gamma \in (0,\frac{1}{2})\), for \(\alpha \in (6/p, 1)\) and $p > 6$, and for \(q, q' \geq 2\), the following space:
\[
\mathfrak{Y}_0 = \LL^q([0,T], \HH^{\alpha}_p(\T^3)) \cap \HH^{\gamma}_{q'}([0,T], \HH^{-2}_{2}(\T^3))
\]
is relatively compact in the following:
\begin{equation}
\label{definition_space_y}
\mathfrak{Y} := C([0,T], \HH^{\eta}_p(\T^3;\R^3)),
\end{equation}
for any $\eta \in (6/p, \alpha)$.
For the sake of clarity, let us state the following two results used to show relative compactness in $C([0,T], \HH^{\eta}_p(\T^3;\R^3))$.

By applying Chebyshev's inequality, we obtain:
\[
\prob \left[\| g^{N}_{\cdot }\|^2_{\mathfrak{Y}_0} > R\right] \leq \frac{\E\left[\|g_{\cdot}^{N}\|^2_{\mathfrak{Y}_0}\right]}{R} \leq \frac{C}{R}
\]
for any \(R > 0\).
As a consequence of this and that $\mathfrak{Y}_0$ is relatively compact in $\mathfrak{Y}$, the sequence of laws of $g^N$ is tight in $\mathfrak{Y}$ .


\section{Characterization of the limit}
\label{sec:passing-to-the-limit}

Let us consider a test function
\(\phi : \mathbb{T}^{3} \to \mathbb{R}^{3}\) that is of class $C^{\infty}$. Since $g^N:\T^3 \to \R^3$ is a smooth function, it also satisfies:
\[
\begin{aligned}
& \left\langle  g^{N}_{t}, \phi \right\rangle = \left\langle g^{N}_{0}, \phi \right\rangle + \int_{0}^{t} \left\langle A g^{N}_{t}, \phi  \right\rangle ds \\
& - \int_{0}^{t} \left\langle \left\langle \mu^{N}_{s}, u^{N}_{s} \cdot \nabla V^{N}(x - \cdot) \right\rangle, \phi  \right\rangle ds + \int_{0}^{t} \left\langle \left\langle \mu^{N}_{s}, V^{N}(x - \cdot) \cdot \nabla u^{N}_{s} \right\rangle , \phi \right\rangle ds \\
& + \int_{0}^{t} \left\langle \left\langle \mu^{N}_{s}, \left( \sigma \cdot \nabla \sigma \right) \cdot \nabla V^{N}(x - \cdot) \right\rangle, \phi \right\rangle ds + \int_{0}^{t} \left\langle \left\langle \mu^{N}_{s}, \left( \left( \nabla \sigma \right) \cdot \nabla \left( \nabla \sigma \right) \right) \cdot \nabla V^{N}(x - \cdot) \right\rangle, \phi \right\rangle ds \\
& - \int_{0}^{t} \left\langle \left\langle \mu^{N}_{s}, \sigma \cdot \nabla V^{N}(x - \cdot) \right\rangle, \phi  \right\rangle dW + \int_{0}^{t} \left\langle \left\langle \mu^{N}_{s}, V^{N}(x - \cdot) \nabla \sigma \right\rangle, \phi  \right\rangle dW \\
& + \frac{1}{N} \sum\limits_{i = 1}^{N}\int_{0}^{t} \left\langle \nabla V^{N}(x-X_{s}^{i, N}) \Phi_{s}^{i, N} h_{0}(X_{0}^{i}), \phi \right\rangle dB_{s}^{i}.
\end{aligned}
\]

Recall that by the definition in equation \eqref{eq:definition_term_g}:
\[
g^N_t(x) = \frac{1}{N} \sum\limits_{i = 1}^N \mathfrak{F}_s^{i, N}(x).
\]

\begin{lem}
\label{lem:conv_prob_1}
The following holds:
\begin{equation}
\begin{aligned}
\lim\limits_{N \to \infty} \mathbb{E}\left[\left|\frac{1}{N} \sum\limits_{i = 1}^{N}\int_{0}^{t} \left\langle \nabla \mathfrak{F}_s^{i, N}(x) , \phi \right\rangle dB_{s}^{i}\right|^2\right] = 0.
\end{aligned}
\end{equation}
In addition, assume $\lim\limits_{N \to \infty} g^N = \omega$ in $C([0,T]; \HH^{\eta}_p(\T^3))$ in probability. Then for every $t \in [0,T]$ the following two limits hold in probability:
\begin{equation}
\lim\limits_{N \to \infty}\int_{0}^{t} \left\langle \left\langle \mu^{N}_{s}, \sigma \cdot \nabla V^{N}(x - \cdot) \right\rangle, \phi \right\rangle dW_s
= \int_0^t \left\langle \omega(s, \cdot) , \sigma \cdot \nabla \phi \right\rangle dW_s,
\end{equation}
\begin{equation}
\lim\limits_{N \to \infty}\int_{0}^{t} \left\langle \left\langle \mu^{N}_{s}, V^{N}(x - \cdot) \cdot \nabla u^{N}_{s} \right\rangle , \phi \right\rangle ds = \int_{0}^{t}  \left\langle \omega_s, \nabla \left( K_{\alpha} \star \omega_s \right)\phi \right\rangle ds.
\end{equation}
\end{lem}

\begin{proof}
First of all, notice that:
\[
\nabla\left(V^N * \phi\right)\left(X_s^{i, N}\right) = V^N * \nabla \phi\left(X_s^{i, N}\right),
\]
and that for all \(x \in \T^2\):
\[
\left| V^N * \nabla \phi (x) \right|^2 \leq \|\nabla \phi\|_{\mathbb{L}^{\infty}}^2 \left| \int_{\T^2} V^N(y) dy \right|^2 = \|\nabla \phi\|_{\mathbb{L}^{\infty}}^2,
\]
and then:
\[
\begin{aligned}
\mathbb{E}\left[\left|\frac{1}{N} \sum\limits_{i = 1}^{N}\int_{0}^{t} \left\langle \nabla \mathfrak{F}_s^{i, N}(x) , \phi \right\rangle dB_{s}^{i}\right|^2\right] & = \frac{1}{N^2} \sum_{i=1}^N \int_0^t \mathbb{E}\left[\left|\nabla\left(V^N * \phi\right)\left(X_s^{i, N}\right) \right|^2 \left| \Phi_{s}^{i, N} h_{0}(X_{0}^{i})\right|^2\right] d s \\
& \leq \frac{t}{N} \cdot \|\nabla \phi\|_{\mathbb{L}^{\infty}}^2 \underset{N \rightarrow \infty}{\longrightarrow} 0.
\end{aligned}
\]
Let us prove now the second statement. Notice that:
\[
\begin{aligned}
\int_{0}^{t} \left\langle \left\langle \mu^{N}_{s}, \sigma \cdot \nabla V^{N}(x - \cdot) \right\rangle, \phi  \right\rangle dW_s & = \int_{0}^{t} \int_{\T^3} \phi(x) \nabla \left( \frac{1}{N}\sum\limits_{i=1}^N \mathfrak{F}_s^{i, N}(x) \sigma(X^{i, N}) \right) dx dW_s \\
& = \int_{0}^{t} \int_{\T^3} \nabla \phi(x) \left( \frac{1}{N}\sum\limits_{i=1}^N \mathfrak{F}_s^{i, N}(x) \sigma(X^{i, N}) \right) dx dW_s \\
&= \int_{0}^{t} \frac{1}{N}\sum\limits_{i=1}^N \left(\int_{\T^3} \nabla \phi(x)  V^{N}(x - X^{i, N}) dx \right) \Phi^{i, N}h_0(X_0^i)\sigma(X^{i, N}) dW_s \\
&= \int_{0}^{t} \frac{1}{N}\sum\limits_{i=1}^N \left(\nabla \phi \star V^N \right)(X^{i,N}) \Phi^{i, N}h_0(X_0^i)\sigma(X^{i, N}) dW_s.
\end{aligned}
\]
Let us denote:
\[
\mathcal{M}_s^N:= \frac{1}{N}\sum\limits_{i=1}^N \left(\nabla \phi \star V^N \right)(X^{i,N}) \Phi^{i, N}h_0(X_0^i)\sigma(X^{i, N}) -\left\langle g^N_s , \left( V^N \star \nabla \phi \right) \cdot \sigma \right\rangle.
\]
With this definition, it is sufficient to show that, for every $\epsilon > 0$:
\begin{equation}
\label{eq:almost_the_limit_term}
\prob \left[ \left|  \int_0^t \mathcal{M}_s^N dW_s\right| > \epsilon \right] \xrightarrow[N \to \infty]{} 0.
\end{equation}
To prove \eqref{eq:almost_the_limit_term}, we notice that for every $\epsilon, K > 0$ (see Formula 1.21 in \cite{nualart}):
\[
\begin{aligned}
&\prob \left[ \left|  \int_{0}^{t} \mathcal{M}_s^N dW_s \right| > \epsilon \right] \leq \prob \left[   \int_{0}^{t} \left| \mathcal{M}_s^N \right|^2 ds > K \right] + \frac{K}{\epsilon^2}.
\end{aligned}
\]
Now, since:
\[
\left\langle g^N_s , \left( V^N \star \nabla \phi \right) \cdot \sigma \right\rangle = \left\langle \mu^N_s , V^N \star \left( \left( V^N \star \nabla \phi \right) \sigma \right)\right\rangle,
\]
one has:
\[
\begin{aligned}
\sup\limits_{s \in [0, T]}\left| \mathcal{M}_s^N \right| \leq \sup\limits_{x \in \T^3} \left| \left( V^N \star \nabla \phi \right) \cdot \sigma^k(x) - \left( \left( V^N \star \nabla \phi \right) \cdot \sigma^k \right) \star V^N (x) \right| \xrightarrow{N \to \infty} 0.
\end{aligned}
\]
Now, let us study the right hand side of the previous inequality:
\begin{equation}
\label{eq:char_limit_1}
\begin{aligned}
& \left| \sigma(x) \cdot \nabla\left(V^N \star \phi\right)(x)-\left( \left(\sigma \cdot \nabla\left(V^N \star \phi\right)\right) \star V^N\right)(x) \right| \\
& \stackrel{\left(\int V^N=1\right)}{\leq} \int_{\T^3} V(y)\left\|\nabla\left(V^N \star \phi\right)(x)\right\|\left\|\sigma(x)-\sigma\left(x-\frac{y}{N^\beta}\right)\right\| d y \\
& +\int_{\T^3} V(y)\left\|\nabla\left(V^N \star \phi\right)(x)-\nabla\left(V^N \star \phi\right)\left(x-\frac{y}{N^\beta}\right)\right\|\left\| \sigma(x)\right\| d y \\
& \leq \int_{\T^3} V(y)\left\|\nabla\left(V^N \star \phi\right)(x)\right\|\left\|\sigma(x)-\sigma\left(x-\frac{y}{N^\beta}\right)\right\| d y +\frac{C}{N^\beta} \int_{\T^3} V(y)\|y\| d y  \leq \\
& \leq \frac{1}{N^{\tilde{\eta} \beta}} \sup _{x, y \in \T^3} \frac{\left\|\sigma(x)-\sigma(y)\right\|}{\|x-y\|^{\tilde{\eta}}} \int_{\T^3} V(y)\|y\|^{\tilde{\eta}} d y + \frac{1}{N^\beta} \int_{\T^3} V(y)\|y\| d y \xrightarrow[N \to \infty]{} 0.
\end{aligned}
\end{equation}
for any $\tilde{\eta} \geq 0$. Therefore, it is enough to fix a $\tilde{\eta}$ such that:
\[
\sup _{x, y \in \T^3} \frac{\left\|\sigma(x)-\sigma(y)\right\|}{\|x-y\|^{\tilde{\eta}}} < \infty
\]

We are left with the last limit.
\[
\begin{aligned}
\lim\limits_{N \to \infty} \int_{0}^{t} \left\langle \left\langle \mu^{N}_{s}, V^{N}(x - \cdot) \cdot \nabla u^{N}_{s} \right\rangle , \phi \right\rangle ds=\int_{0}^{t}  \left\langle \mu^{N}_{s}, \left( \phi \star V^{N}\right) \cdot \nabla \left(K_{\alpha} \star g^N_s \right) \right\rangle ds.
\end{aligned}
\]
Now, one gets:
\begin{equation}
\begin{aligned}
& \left| \left\langle \mu^{N}_{s}, \left( \phi \star V^{N}\right) \cdot \nabla \left(K_{\alpha} \star g^N_s \right) \right\rangle -\left\langle g^{N}_{s}, \left( \phi \star V^{N}\right) \cdot \nabla \left(K_{\alpha} \star g^N_s \right) \right\rangle \right| \\
& \quad \leq \sup_{x \in \mathbb{T}^3} \left| \left( \phi \star V^{N}\right) \cdot \nabla \left(K_{\alpha} \star g^N_s \right)(x) - V^N \star \left( \left( \phi \star V^{N}\right) \cdot \nabla \left(K_{\alpha} \star g^N_s \right) \right)(x)  \right|
\end{aligned}
\end{equation}

Similarly to equation \eqref{eq:char_limit_1}, we have:
\[
\begin{aligned}
& \sup_{x \in \mathbb{T}^3} \left| \left( \phi \star V^{N}\right) \cdot \nabla \left(K_{\alpha} \star g^N_s \right)(x) - V^N \star \left( \left( \phi \star V^{N}\right) \cdot \nabla \left(K_{\alpha} \star g^N_s \right) \right)(x)  \right| \\
& \leq \frac{1}{N^{\tilde{\eta} \beta}} \sup _{x, y \in \T^3} \frac{\left\| \nabla K_{\alpha} \star g^N_s(x)- \nabla K_{\alpha} \star g^N_s(y)\right\|}{\|x-y\|^{\tilde{\eta}}} \int_{\T^3} V(y)\|y\|^{\tilde{\eta}} d y + \frac{1}{N^\beta} \int_{\T^3} V(y)\|y\| d y.
\end{aligned}
\]
If $\tilde{\eta}:= s - 3/r$ then for any $q \geq 2$:
\[
\E\left[ \left\| \nabla \left(K_{\alpha} \star g^N_s \right) \right\|_{C^{\tilde{\eta}}(\T^3)}^q \right] \leq C.
\]
In fact, because of Sobolev embeddings, there exists $s, r$ and $p>6$ such that:
\[
\begin{aligned}
\left\| \nabla \left(K_{\alpha} \star g^N_s \right) \right\|_{C^{s - 3/r}(\T^3)} \leq \left\| \nabla \left(K_{\alpha} \star g^N_s \right) \right\|_{\HH^{s}_r(\T^3)} \leq \left\| \nabla \left(\nabla \left(K_{\alpha} \star g^N_s \right) \right) \right\|_{\LL^r(\T^3)} \leq \left\| g^N_s\right\|_{\LL^p(\T^3)}.
\end{aligned}
\]
The last inequality is a consequence of Proposition \ref{th:continuity_of_K_operator}.
Therefore, because of Proposition \ref{thm:compact_1}:
\[
\E \left[ \left\| g^N_s\right\|_{\LL^p(\T^3)}^q \right] \leq C,
\]
and this concludes.
\end{proof}

\begin{prop}
If $g^N \xrightarrow[N \to \infty]{} \omega$ in $C([0,T]; \HH^{\eta}_p(\T^3))$ in probability, then $\omega$ satisfies the following:
\begin{equation}
\label{eq:stoch_navier_stokes}
\begin{aligned}
& \left\langle  \omega_{t}, \phi \right\rangle = \left\langle \omega_{0}, \phi \right\rangle + \int_{0}^{t} \left\langle \omega_{s}, A \phi  \right\rangle ds \\
&- \int_{0}^{t} \left\langle \omega_{s}, \left( K_{\alpha} \star \omega_{s} \right) \nabla \phi  \right\rangle ds + \int_{0}^{t} \left\langle \omega_{s}, \nabla \left( K_{\alpha} \star \omega_{s} \right) \phi  \right\rangle ds \\
& - \int_{0}^{t} \left\langle \omega_s , \left( \sigma \cdot \nabla \sigma \right) \cdot \nabla \phi \right\rangle ds + \int_{0}^{t} \left\langle \omega_s , \left( \left( \nabla \sigma \right) \cdot \nabla \left( \nabla \sigma \right) \right) \cdot \nabla \phi \right\rangle ds \\
& - \int_{0}^{t} \left\langle \omega_s , \sigma \cdot \nabla \phi  \right\rangle dW_s + \int_{0}^{t} \left\langle \omega_s , \nabla \sigma \phi  \right\rangle dW_s,\\
\end{aligned}
\end{equation}
where $A: \HH^{2}_p(\T^3) \to \LL^p(\T^3)$ is defined as:
\[
\left\langle A, f \right\rangle := \nu \Delta f + \frac{1}{2} \sigma^T (Hf) \sigma + \frac{1}{2} \tr \left( \left( \nabla \sigma \right)^{T}(Hf)\left( \nabla \sigma \right) \right),
\]
\end{prop}

\begin{proof}
This Proposition is just a consequence of Lemma \ref{lem:conv_prob_1}.
\end{proof}

\section{Well-posedness of the Navier-Stokes-alpha model}
\label{sec:well-posed-limit}

\begin{remark}
Following from Remark \ref{rem:remark_boundedness}, if $\omega$ is a limit point of $g^N$, then:
\[
\E \left[ \int_0^t\norm{\omega}_{\LL^2(\T^3)}^2 ds\right] < +\infty.
\]
\end{remark}

\begin{prop}
\label{thm:uniqueness}
Let $\omega^1$ and $\omega^2$ be two stochastic processes with values in $C([0,T]; \HH^{\eta}_p(\T^3; \R^3))$ such that $\omega^1_0 = \omega^2_0$.
Let us assume:
\[
\E \left[ \int_0^t\norm{\omega^i}_{\LL^2(\T^3)}^2 ds\right] < +\infty, \quad \text{for } i=1,2,
\]
and that for every $\phi \in C^{\infty}(\T^3)$ they both satisfy the following:
\begin{equation}
\label{eq:solu}
\begin{aligned}
& \left\langle  \omega_{t}, \phi \right\rangle = \left\langle \omega_{0}, \phi \right\rangle + \int_{0}^{t} \left\langle \omega_{s}, A \phi  \right\rangle ds \\
&- \int_{0}^{t} \left\langle \omega_{s}, \left( K_{\alpha} \star \omega_{s} \right) \nabla \phi  \right\rangle ds + \int_{0}^{t} \left\langle \omega_{s}, \nabla \left( K_{\alpha} \star \omega_{s} \right) \phi  \right\rangle ds \\
& - \int_{0}^{t} \left\langle \omega_s , \left( \sigma \cdot \nabla \sigma \right) \cdot \nabla \phi \right\rangle ds + \int_{0}^{t} \left\langle \omega_s , \left( \left( \nabla \sigma \right) \cdot \nabla \left( \nabla \sigma \right) \right) \cdot \nabla \phi \right\rangle ds \\
& - \int_{0}^{t} \left\langle \omega_s , \sigma \cdot \nabla \phi  \right\rangle dW_s + \int_{0}^{t} \left\langle \omega_s , \nabla \sigma \phi  \right\rangle dW_s .\\
\end{aligned}
\end{equation}
Then set such that $\omega^1_t = \omega^2_t$ for all $t \in [0, T]$ has probability $1$ (they are indistinguishable).
\end{prop}
\begin{proof}
Let us define, for $k \in \Z$, the following $e_k:\T^3 \to \mathbb{C}$:
\[
e_k(x) := e^{ik\cdot x}, \quad \text{ for } x \in \T^3.
\]
Let us denote $\tilde{\omega}:=\omega^1-\omega^2$. Applying Ito's formula to:
\[
\left| \left\langle \tilde{\omega}_t, e_k\right\rangle \right|^2 = \left\langle \tilde{\omega}_t, e_k\right\rangle \left\langle \tilde{\omega}_t, e_{-k}\right\rangle,
\]
one gets:
\[
\begin{aligned}
\left| \left\langle \tilde{\omega}_t, e_k\right\rangle \right|^2 = \sum\limits_{j = 1}^{10} \int_0^t J^{k}_j(s) ds + \sum\limits_{j = 11}^{12} \int_0^t J^{k}_j(s) dW_s,
\end{aligned}
\]
where:
\[
J^k_1(s) = \left\langle \tilde{\omega}_s, e_{-k}\right\rangle \left\langle \tilde{\omega}_s, A e_k\right\rangle + \left\langle \tilde{\omega}_s, A e_{-k}\right\rangle \left\langle \tilde{\omega}_s, e_k\right\rangle,
\]
\[
J^k_2(s) = \left\langle \tilde{\omega}_s, e_{-k}\right\rangle \left\langle \omega^1_s, \left( K_{\alpha} \star \omega^1_s \right)
 \nabla e_k\right\rangle + \left\langle \tilde{\omega}_s, e_{k}\right\rangle \left\langle \omega^1_s, \left( K_{\alpha} \star \omega^1_s \right)
 \nabla e_{-k}\right\rangle,
\]
\[
J^k_3(s) = -\left\langle \tilde{\omega}_s, e_{-k}\right\rangle \left\langle \omega^2_s, \left( K_{\alpha} \star \omega^2_s \right)
\nabla e_k\right\rangle - \left\langle \tilde{\omega}_s, e_{k}\right\rangle \left\langle \omega^2_s, \left( K_{\alpha} \star \omega^2_s \right)
 \nabla e_{-k}\right\rangle,
\]
\[
J^k_4(s) = \left\langle \tilde{\omega}_s, e_{-k}\right\rangle \left\langle \omega^1_s, \nabla \left( K_{\alpha} \star \omega^1_s \right) e_k\right\rangle + \left\langle \tilde{\omega}_s, e_{k}\right\rangle \left\langle \omega^1_s, \nabla\left( K_{\alpha} \star \omega^1_s \right) e_{-k}\right\rangle,
\]
\[
J^k_5(s) = -\left\langle \tilde{\omega}_s, e_{-k}\right\rangle \left\langle \omega^2_s, \nabla \left( K_{\alpha} \star \omega^2_s \right) e_k\right\rangle - \left\langle \tilde{\omega}_s, e_{k}\right\rangle \left\langle \omega^2_s, \nabla\left( K_{\alpha} \star \omega^2_s \right) e_{-k}\right\rangle,
\]
\[
J^k_6(s) = \left\langle \tilde{\omega}_s , e_{-k}  \right\rangle \left\langle \tilde{\omega}_s , \left( \sigma \cdot \nabla \sigma \right) \cdot \nabla e_{k}  \right\rangle + \left\langle \tilde{\omega}_s , e_{k}  \right\rangle \left\langle \tilde{\omega}_s , \left( \sigma \cdot \nabla \sigma \right) \cdot \nabla e_{-k}  \right\rangle
\]
\[
J^k_7(s) = \left\langle \tilde{\omega}_s , e_{-k}  \right\rangle \left\langle \tilde{\omega}_s , \left( \left( \nabla \sigma \right) \cdot \nabla \left( \nabla \sigma \right) \right) \cdot \nabla e_{k}  \right\rangle + \left\langle \tilde{\omega}_s , e_{k}  \right\rangle \left\langle \tilde{\omega}_s , \left( \left( \nabla \sigma \right) \cdot \nabla \left( \nabla \sigma \right) \right) \cdot \nabla e_{-k}  \right\rangle
\]
\[
J^k_8(s) = \left| \left\langle \tilde{\omega}_s , \sigma \cdot \nabla e_k  \right\rangle \right|^2
\]
\[
J^k_9(s) = \left| \left\langle \tilde{\omega}_s ,  e_k \cdot \nabla \sigma \right\rangle \right|^2
\]
\[
J^k_{10}(s) = \left\langle \tilde{\omega}_s ,  e_{-k} \cdot \nabla \sigma \right\rangle \left\langle \tilde{\omega}_s , \sigma \cdot \nabla e_k  \right\rangle + \left\langle \tilde{\omega}_s ,  e_{k} \cdot \nabla \sigma \right\rangle \left\langle \tilde{\omega}_s , \sigma \cdot \nabla e_{-k}  \right\rangle
\]
\[
J^k_{11}(s) = \left\langle \tilde{\omega}_s , e_{-k}  \right\rangle \left\langle \tilde{\omega}_s , \sigma \cdot \nabla e_k  \right\rangle + \left\langle \tilde{\omega}_s , e_{k}  \right\rangle \left\langle \tilde{\omega}_s , \sigma \cdot \nabla e_{-k}  \right\rangle
 \]
 \[
 J^k_{12}(s) = \left\langle \tilde{\omega}_s , e_{-k}  \right\rangle \left\langle \tilde{\omega}_s , e_k \cdot \nabla \sigma  \right\rangle + \left\langle \tilde{\omega}_s , e_{k}  \right\rangle \left\langle \tilde{\omega}_s , e_{-k} \cdot \nabla \sigma  \right\rangle
 \]
Let $\alpha>0$, and let us compute the $\HH^{-\alpha}_2(\T^3)$-norm of $\tilde{\omega}$. One has:
\[
\begin{aligned}
\left\| \tilde{\omega}\right\|_{\HH^{-\alpha}_2(\T^3)}^2 =& \sum\limits_{k \in \Z^3} \frac{1}{(1+|k|^2)^{\alpha / 2}} \left| \left\langle \tilde{\omega}_t, e_k\right\rangle \right|^2 =
\sum\limits_{j=1}^{10}\int_0^t \sum\limits_{k \in \Z^3} \frac{1}{(1+|k|^2)^{\alpha / 2}}J^k_j(s) ds \\
+& \sum\limits_{j=11}^{12}\int_0^t \sum\limits_{k \in \Z^3} \frac{1}{(1+|k|^2)^{\alpha / 2}}J^k_j(s) dW_s.
\end{aligned}
\]
Now, first of all:
\[
\sum\limits_{k \in \Z^3} \frac{1}{(1+|k|^2)^{\alpha / 2}}J^k_1(s) \leq C(\sigma)\sum\limits_{k \in \Z^3} \frac{(ik)^2}{(1+|k|^2)^{\alpha / 2}}\left\langle \tilde{\omega}_s, e_{-k}\right\rangle \left\langle \tilde{\omega}_s, e_k\right\rangle = -C(\sigma) \sum\limits_{k \in \Z^3} \frac{k^2}{(1+|k|^2)^{\alpha / 2}} \left| \left\langle \tilde{\omega}_s, e_{k}\right\rangle \right|^2.
\]
This is because:
\[
\tr \left( (\nabla \sigma)^T He_k (\nabla \sigma ) \right) = i^2 \tr \left( (\nabla \sigma)^T k k^T \nabla \sigma \right) = - \norm{(\nabla \sigma)^T k} \leq 0.
\]
Now,
\[
\begin{aligned}
\sum\limits_{k \in \Z^3} \frac{1}{(1+|k|^2)^{\alpha / 2}}(J^k_2(s) + J^k_3(s) ) & = \sum\limits_{k \in \Z^3} \frac{1}{(1+|k|^2)^{\alpha / 2}}\left\langle \tilde{\omega}_s, e_{-k}\right\rangle \left\langle \omega^1_s, \left( K_{\alpha} \star \tilde{\omega}_s \right)
 \nabla e_k\right\rangle \\
&+ \sum\limits_{k \in \Z^3} \frac{1}{(1+|k|^2)^{\alpha / 2}}\left\langle \tilde{\omega}_s, e_{-k}\right\rangle \left\langle \tilde{\omega}_s, \left( K_{\alpha} \star \omega^2_s \right)
 \nabla e_k\right\rangle \\
 &= \sum\limits_{k \in \Z^3} \frac{ik}{(1+|k|^2)^{\alpha / 2}}\left\langle \tilde{\omega}_s, e_{-k}\right\rangle \left\langle \omega^1_s \left( K_{\alpha} \star \tilde{\omega}_s \right), e_k\right\rangle \\
&+ \sum\limits_{k \in \Z^3} \frac{ik}{(1+|k|^2)^{\alpha / 2}}\left\langle \tilde{\omega}_s, e_{-k}\right\rangle \left\langle \tilde{\omega}_s \left( K_{\alpha} \star \omega^2_s \right), e_k \right\rangle \\
&= \sum\limits_{k \in \Z^3} \frac{ik}{(1+|k|^2)^{\alpha / 2}}\left\langle \tilde{\omega}_s, e_{-k}\right\rangle \left( \left\langle \omega^1_s \left( K_{\alpha} \star \tilde{\omega}_s \right), e_k\right\rangle + \left\langle \tilde{\omega}_s \left( K_{\alpha} \star \omega^2_s \right), e_k \right\rangle \right)
\end{aligned}
\]
Furthermore:
\[
\begin{aligned}
\sum\limits_{k \in \Z^3} \frac{1}{(1+|k|^2)^{\alpha / 2}}(J^k_6(s) + J^k_7(s) ) &= \sum\limits_{k \in \Z^3} \frac{1}{(1+|k|^2)^{\alpha / 2}} \left\langle \tilde{\omega}_s , e_{-k}  \right\rangle \left\langle \tilde{\omega}_s , \left( \sigma \cdot \nabla \sigma \right) \cdot \nabla e_{k}  \right\rangle \\
&+ \sum\limits_{k \in \Z^3} \frac{1}{(1+|k|^2)^{\alpha / 2}} \left\langle \tilde{\omega}_s , e_{-k}  \right\rangle \left\langle \tilde{\omega}_s , \left( \left( \nabla \sigma \right) \cdot \nabla \left( \nabla \sigma \right) \right) \cdot \nabla e_{k}  \right\rangle \\
&= \sum\limits_{k \in \Z^3} \frac{ik}{(1+|k|^2)^{\alpha / 2}} \left\langle \tilde{\omega}_s , e_{-k}  \right\rangle \left\langle \tilde{\omega}_s \left( \sigma \cdot \nabla \sigma \right), e_{k}  \right\rangle \\
&+ \sum\limits_{k \in \Z^3} \frac{ik}{(1+|k|^2)^{\alpha / 2}} \left\langle \tilde{\omega}_s , e_{-k}  \right\rangle \left\langle \tilde{\omega}_s \left( \left( \nabla \sigma \right) \cdot \nabla \left( \nabla \sigma \right) \right) , e_{k}  \right\rangle \\
&= \sum\limits_{k \in \Z^3} \frac{ik}{(1+|k|^2)^{\alpha / 2}} \left\langle \tilde{\omega}_s , e_{-k}  \right\rangle \left\langle \tilde{\omega}_s \left( \left( \nabla \sigma \right) \cdot \nabla \left( \nabla \sigma \right)  + \left( \sigma \cdot \nabla \sigma \right) \right), e_{k}  \right\rangle
\end{aligned}
\]
Now:
\[
\sum\limits_{k \in \Z^3} \frac{1}{(1+|k|^2)^{\alpha / 2}}J^k_8(s) = -\sum\limits_{k \in \Z^3} \frac{k^2}{(1+|k|^2)^{\alpha / 2}} \left| \left\langle \tilde{\omega}_s \cdot \sigma, e_{k}\right\rangle \right|^2
\]
Finally:
\[
\begin{aligned}
\sum\limits_{k \in \Z^3} \frac{1}{(1+|k|^2)^{\alpha / 2}}J^k_{10}(s) & = \sum\limits_{k \in \Z^3} \frac{ik}{(1+|k|^2)^{\alpha / 2}} \left\langle \tilde{\omega}_s \sigma, e_k  \right\rangle \left\langle \tilde{\omega}_s ,  e_{-k} \cdot \nabla \sigma \right\rangle \\
&\leq \frac{1}{2} \sum\limits_{k \in \Z^3} \frac{k^2}{(1+|k|^2)^{\alpha / 2}} \left| \left\langle \tilde{\omega}_s \sigma, e_k  \right\rangle \right|^2 + \frac{1}{2} \sum\limits_{k \in \Z^3} \frac{1}{(1+|k|^2)^{\alpha / 2}} \left| \left\langle \tilde{\omega}_s ,  e_{-k} \cdot \nabla \sigma \right\rangle \right|^2 \\
&= \frac{1}{2} \norm{\tilde{\omega}_s \sigma}_{\HH^{2-\alpha}_2(\T^3)}^2 + \frac{1}{2} \norm{\tilde{\omega}_s \cdot \nabla \sigma}_{\HH^{-\alpha}_2(\T^3)}^2
\end{aligned}
\]
Finally one has:
\begin{equation}
\label{eq:estimate_uniqueness}
\begin{aligned}
\sum\limits_{k \in \Z^3}\frac{1}{(1+|k|^2)^{\alpha / 2}} \sum\limits_{j = 1}^{10} J^{k}_j(s) & \leq -\norm{\tilde{\omega}_s}_{\HH^{2-\alpha}_2(\T^3)}^2 + \frac{1}{2} \norm{\tilde{\omega}_s}_{\HH^{2-\alpha}_2(\T^3)}^2 \\
&- \norm{\tilde{\omega}_s \sigma }_{\HH^{2-\alpha}_2(\T^3)}^2 + \frac{1}{2} \norm{\tilde{\omega}_s \sigma }_{\HH^{2-\alpha}_2(\T^3)}^2 \\
&+ C \bigg( \norm{\omega^1_s K_{\alpha} \star \tilde{\omega}_s}_{\HH^{-\alpha}_2(\T^3)}^2 + \norm{\tilde{\omega}_s K_{\alpha} \star \omega^2_s}_{\HH^{-\alpha}_2(\T^3)}^2 + \norm{\tilde{\omega}_s \sigma \cdot \nabla \sigma}_{\HH^{-\alpha}_2(\T^3)}^2 \\
&+ \norm{\tilde{\omega}_s \left( \left( \nabla \sigma \right) \cdot \nabla \left( \nabla \sigma \right) \right)}_{\HH^{-\alpha}_2(\T^3)}^2 +  \norm{\omega^1_s \nabla K_{\alpha} \star \tilde{\omega}_s}_{\HH^{-\alpha}_2(\T^3)}^2 + \norm{\tilde{\omega}_s \nabla K_{\alpha} \star \omega^2_s}_{\HH^{-\alpha}_2(\T^3)}^2 \bigg).
\end{aligned}
\end{equation}
Let us estimate now:
\[
\norm{\tilde{\omega}_s \nabla K_{\alpha} \star \omega^2_s}_{\HH^{-\alpha}_2(\T^3)},
\]
and
\[
\norm{\omega^1_s \nabla K_{\alpha} \star \tilde{\omega}_s}_{\HH^{-\alpha}_2(\T^3)}.
\]
For the first:
\[
\norm{\tilde{\omega}_s \nabla K_{\alpha} \star \omega^2_s}_{\HH^{-\alpha}_2(\T^3)} \leq \norm{\tilde{\omega}_s}_{\HH^{-\alpha}_2(\T^3)} \norm{ \nabla K_{\alpha} \star \omega^2_s}_{\HH^{\alpha}_2(\T^3)} \leq C \norm{\tilde{\omega}_s}_{\HH^{-\alpha}_2(\T^3)} \norm{\omega^2_s}_{\LL^{2}(\T^3)},
\]
where the second inequality holds if $0< \alpha < 1$. In fact:
\[
\begin{aligned}
\norm{ \nabla K_{\alpha} \star \omega^2_s}_{\HH^{\alpha}_2(\T^3)}^2 &= \sum\limits_{k \in \Z^3-\{0\}} \frac{(1+|k|^2)^{(1+\alpha)/2}|k|}{(1+|k|^2)|k|^3} |\left\langle \omega_s^2, e_k\right\rangle|^2 \\
&\leq C \sum\limits_{k \in \Z^3-\{0\}} (1+|k|^2)^{(-3+\alpha)/2} |\left\langle \omega_s^2, e_k\right\rangle|^2 \leq C \norm{ \omega^2_s}_{\LL^{2}(\T^3)}^2.
\end{aligned}
\]
For the second, since:
\[
\begin{aligned}
\left| \left\langle \omega^1_{s}, \nabla \left( K_{\alpha} \star \tilde{\omega}_{s} \right) \phi  \right\rangle \right| & \leq \left\| \omega^1_{s} \right\|_{\LL^2}\left\| \nabla \left( K_{\alpha} \star \tilde{\omega}_{s} \right) \phi  \right\|_{\LL^2} \leq \left\| \omega^1_{s} \right\|_{\LL^2}\left\| \nabla \left( K_{\alpha} \star \tilde{\omega}_{s} \right) \phi  \right\|_{\LL^2} \\
& \leq \left\| \omega^1_{s} \right\|_{\LL^2}\left\| \nabla \left( K_{\alpha} \star \tilde{\omega}_{s} \right) \right\|_{\LL^2} \left\| \phi \right\|_{\LL^{\infty}}\leq C \left\| \omega^1_{s} \right\|_{\LL^2} \left\| \tilde{\omega}_s\right\|_{\HH^{-\alpha}_2(\T^3)},
\end{aligned}
\]
where for the penultimate inequality we used the same argument used in the first estimate.
Finally, we obtained:
\[
\norm{\tilde{\omega}_t}_{\HH^{-\alpha}_2(\T^3)}^2 \leq \int_0^t  \left\| \omega^1_{s} \right\|_{\LL^2}^2\norm{\tilde{\omega}_s}_{\HH^{-\alpha}_2(\T^3)}^2 ds + M_t,
\]
where:
\[
M_t = \int_0^t\sum\limits_{k \in \Z^3} \frac{1}{(1+|k|^2)^{\alpha/2}}J^k_{11}(s) dW_s + \int_0^t\sum\limits_{k \in \Z^3} \frac{1}{(1+|k|^2)^{\alpha/2}} J^k_{12}(s) dW_s
\]
Let us assume, for the moment, that we have already established that \( M_t \) is a martingale. Let us define the following stopping time:
 \[
 \tau_k := \inf \left\{ t \geq 0: \int_0^t  \left\| \omega^1_{s} \right\|_{\LL^2}^2 \geq k \right\} \wedge T.
 \]
Notice that:
\[
\tau_k \uparrow T, \text{ as } k \to +\infty.
\]
In fact, let us assume that there exists a set of positive probability $E$ such that, in $E$, the following holds:
\[
\limsup\limits_{k \in \N} \tau_k < T.
\]
This implies, in particular, that for each event in $E$, there exists a $t < T$ such that:
\[
\int_0^t  \left\| \omega^1_{s} \right\|_{\LL^2}^2 \geq k, \quad \text{ for all } k \in \N.
\]
In other words:
\[
\int_0^T  \left\| \omega^1_{s} \right\|_{\LL^2}^2 = +\infty,
\]
in a set of positive probability ($E$).
But this is in contrast with the assumption:
\[
\mathbb{E} \left[ \int_0^T  \left\| \omega^1_{s} \right\|_{\LL^2}^2\right] < +\infty.
\]
Now, by applying Lemma 5.3 of \cite{MR2502705} one has:
\[
\E \left[ \sup\limits_{t \in [0, \tau_k]}\norm{\tilde{\omega}_t}_{\HH^{-\alpha}_2(\T^3)}^2 \right] = 0,
\]
and letting $k \to +\infty$ one has:
\[
\E \left[ \sup\limits_{t \in [0, T]}\norm{\tilde{\omega}_t}_{\HH^{-\alpha}_2(\T^3)}^2 \right] = 0,
\]
from which the thesis follows.
Let us prove, now, that $M_t$ is a martingale. Let us start showing that:
\[
\sum\limits_{k \in \Z^3} \frac{1}{(1+|k|^2)^{\alpha/2}}J^k_{11}(s),
\]
is square integrable. One has:
\[
\begin{aligned}
\left( \sum\limits_{k \in \Z^3} \frac{1}{(1+|k|^2)^{\alpha/2}}J^k_{11}(s) \right)^2 &\leq \left( \sum\limits_{k \in \Z^3} \frac{1}{(1+|k|^2)^{\alpha/2}} \left\langle \tilde{\omega}_s , e_{-k}  \right\rangle \left\langle \tilde{\omega}_s , \sigma \cdot \nabla e_k  \right\rangle \right)^2 \\
&\leq C \left( \sum\limits_{k \in \Z^3} \frac{k^2}{(1+|k|^2)^{2\alpha/2}} |\left\langle \tilde{\omega}_s , e_{-k}  \right\rangle|^2 \right) \left( \sum\limits_{k \in \Z^3} |\left\langle \sigma \tilde{\omega}_s , e_{k}  \right\rangle|^2 \right) \\
&= C \norm{\tilde{\omega}_s}_{\HH^{-2\alpha + 2}_2}^2 \norm{\tilde{\omega}_s}_{\LL^{2}}^2 \leq \frac{C}{2}\norm{\tilde{\omega}_s}_{\HH^{-2\alpha + 2}_2}^4 +\frac{C}{2}\norm{\tilde{\omega}_s}_{\LL^{2}}^4.
\end{aligned}
\]
Now, since for any $\eta>0$ there exists $0<\alpha<1$ such that:
\[
\HH^{\eta}_p(\T^3) \subset \HH^{-2\alpha + 2}_2(\T^3),
\]
one concludes that:
\[
\E \left[ \int_0^t \left( \sum\limits_{k \in \Z^3} \frac{1}{(1+|k|^2)^{\alpha/2}}J^k_{11}(s) \right)^2 ds\right] < +\infty.
\]
Similarly, one obtains:
\[
\E \left[ \int_0^t \left( \sum\limits_{k \in \Z^3} \frac{1}{(1+|k|^2)^{\alpha/2}}J^k_{12}(s) \right)^2 ds \right] < +\infty.
\]
\end{proof}

\section{Proof of the main result}
\label{sec:conclusions}
To prove the Theorem, we make use of the following Gyongy-Krylov Lemma (see  \cite{gyongy} for a proof):

\begin{lem}
\label{lemma:gyongy}
Let $\left\{X_n\right\}$ be a sequence of random elements in a Polish space $E$ equipped with the Borel $\sigma$-algebra. Then $X_n$ converges in probability to a $E$-valued random element if and only if for each pair $\left(X_{\ell}, X_m\right)$ of subsequences, there exists a subsequence $\left\{v_k\right\}$ given by
\[
v_k=\left(X_{\ell(k)}, X_{m(k)}\right)
\]
converging weakly to a random element $v(x, y)$ supported on the diagonal set
\[
\{(x, y) \in E \times E: x=y\}.
\]
\end{lem}

Here's the main Theorem:

\begin{thm}
\label{thm:main_theorem}
Under the assumption \ref{ass:assprinci} the following limits:
\[
\lim\limits_{N \to \infty} g^N = \omega,
\]
holds in probability in $C([0, T]; \HH^{\eta}_p(\T^3))$ where $\omega$ satisfies equation \eqref{eq:NS_alpha_model}.
\end{thm}

\begin{proof}
Let us consider a pair $v_{i, j} = (g^i, g^j)_{i, j}$ with values in $\mathfrak{Y} \times \mathfrak{Y}$. Because of Section \ref{sec:compactness}, we can conclude that there exists a subsequence of the pair, say $v_{i_k, j_k} = (g^{i_k}, g^{j_k})_k$, such that $v_{i_k, j_k}$ converges weakly to a probability $\nu$ on $\mathfrak{Y} \times \mathfrak{Y}$.

Now, because of Skorokhod representation theorem (see p.70 in \cite{skorohod}), modulo changing probability space, a random variable $(\tilde{u}, \overline{u})$ exists in the new probability space, and $(\overline{g}^{i_k}, \overline{g}^{j_k})_k$ such that:
\[
(\overline{g}^{i_k}, \overline{g}^{j_k})_k \xrightarrow[]{\text{weakly}} (\tilde{u}, \overline{u}),
\]
with $\text{Law}\left(\overline{g}^{i_k}, \overline{g}^{j_k} \right) = \text{Law}\left( g^{i_k}, g^{j_k} \right)$. Therefore, by Section \ref{sec:passing-to-the-limit} one concludes that $\tilde{u}$ and $\overline{u}$ both satisfy equation \eqref{eq:solu}. Now, because of Proposition \ref{thm:compact_1} one can use Proposition \ref{thm:uniqueness} to conclude $\tilde{u} = \overline{u}$ a.s.

Now, we notice that this implies:
\[
\nu \left( \left\{(x, y) \in \mathfrak{Y} \times \mathfrak{Y}: x=y \right\} \right) = 0.
\]
Now Lemma \ref{lemma:gyongy} implies that the original sequence in the original probability space converges in probability.
\end{proof}

\newpage
\section{Appendix}
\label{sec:appendix}

\subsection{Properties of the smoothed kernel}
\label{subsec:properties_of_K}

Let $\T^3$ be the $3$-dimensional torus. We consider the Biot-Savart kernel:
\begin{equation}
\label{eq:biot_savart}
K(x):=\frac{i}{2 \pi} \sum_{k \in \mathbb{Z}^3: k \neq 0} \exp (i k \cdot x) \frac{k}{k^3}, \quad x \in \mathbb{T}^3,
\end{equation}
that acts on a function $\omega:\T^3 \to \R^3$ as follows:
\[
K \star \omega (x) := \int_{\T^3} K(x - y) \wedge \omega(y) dy .
\]
Then we consider the mollified Biot-Savart kernel $K_{\alpha}$:
\[
K_{\alpha} \left( \omega \right) := (\indicator - \alpha^2 \Delta)^{-1} (K \star \omega),
\]
that can be interpreted as a convolution with the kernel:
\[
K_{\alpha}(x) := \frac{i}{2 \pi} \sum_{k \in \mathbb{Z}^3: k \neq 0} \exp (i k \cdot x) \frac{1}{k^3(1+\alpha^2 k^2 )}k.
\]
Therefore we can write:
\[
K_{\alpha} \left( \omega \right) = K_{\alpha} \star \omega, \quad \text{ for }\omega: \T^3 \to \R^3, \quad \omega \in \LL^p(\T^3).
\]
Notice that $K_{\alpha} \in \LL^{\infty}$. This is because:
\[
\left|K_{\alpha}(x)\right| \leq \frac{1}{2 \pi} \sum_{k \in \mathbb{Z}^3: k \neq 0} \left| \frac{1}{k^3(1+\alpha^2 k^2 )}k \right| = C \sum_{h \in \mathbb{Z}: k \neq 0} \left| \frac{h^2}{h^3(1+\alpha^2 h^2 )}h \right| < \infty
\]

\begin{prop}
\label{th:continuity_of_K_operator}
Let $r,q>1$ be such that the following relationship holds:
\[
\frac{1}{r} > \frac{1}{q} + \frac{1}{2}.
\]
Now, if $f \in \LL^q(\T^3)$ then:
\[
\left\| \nabla \left( \nabla \left( K_{\alpha} \star f\right) \right) \right\|_{\LL^{r}(\T^3)} \leq C \left\| f\right\|_{\LL^q(\T^3)}
\]
\end{prop}
\begin{proof}
Let us first notice that:
\[
\begin{aligned}
\left\| \nabla \left( \nabla \left( K_{\alpha} \star f\right) \right) \right\|_{\LL^{r}(\T^3)} & \leq C \left\| (\indicator - \alpha^2\Delta)^{1/2} \left( (\indicator - \alpha^2\Delta)^{1/2} \left( K_{\alpha} \star f\right) \right) \right\|_{\LL^{r}(\T^3)} \\
&=C\left\| (\indicator - \alpha^2\Delta)^{1}  \left( K_{\alpha}\right) \star f \right\|_{\LL^{r}(\T^3)} \\
& =  C\left\| K \star f \right\|_{\LL^{r}(\T^3)}.
\end{aligned}
\]
The first inequality holds because of the definition of $\HH^{1/2}_p$. Note that there exists $p < \frac{3}{2}$ such that:
\[
1+\frac{1}{r} = \frac{1}{p} + \frac{1}{q}.
\]
In fact, $p = \frac{1}{1+1/r-1/q} < 3/2$ if and only if $1/r-1/q > 1/2$.
Now, recalling:
\[
\begin{aligned}
K \star f (x) = \int_{\T^3}K(x-y)f(y)dy,
\end{aligned}
\]
and using Young inequality:
\[
\left\| K \star f \right\|_{\LL^r(\T^3)} \leq \left\|K\right\|_{\LL^p(\T^3)} \left\|f\right\|_{\LL^q(\T^3)} \leq C\left\|f\right\|_{\LL^q(\T^3)}.
\]
Here we used that for any $p < \frac{3}{2}$, the non-smoothed Biot-Savart kernel satisfies:
\[
K \in \LL^{p}(\T^3).
\]
See Chapter 4, Theorem 4.13 in \cite{MR1636569} for a proof that $K \in \LL^{p}(\T^3)$ when $p < \frac{3}{2}$.
\end{proof}

\subsection{Properties of the operator $A$}
Since:
\[
\div\left(\sigma \sigma^T + (\nabla \sigma) (\nabla \sigma)^T\right) = 0,
\]
the operator $A$ can be written in the following way:
\[
\left\langle A, f \right\rangle = \div(M \cdot \nabla f),
\]
for $M:\T^3 \to \R^{3 \times 3}$ defined as:
\[
M := \nu \text{Id} + \frac{1}{2}\sigma \sigma^T + \frac{1}{2}(\nabla \sigma) (\nabla \sigma)^T.
\]
We remark that:
\[
\div(M)=0,
\]
and
\[
\nu' \text{Id} \leq M \leq \nu'' \text{Id}.
\]
Now, as a consequence of Theorem 1 in \cite{sikora}, this tells us that $A$ is an operator with bounded imaginary part. Therefore if if $0<\varepsilon< m $, $1<p<+\infty$, and $\theta \in (0,1)$ the following two hold:
\begin{equation}
\label{eq:domain_A}
[X,D(A^{\alpha})]_{\theta}=D(A^{\theta \alpha}),
\end{equation}
and
\begin{equation}
\label{eq:sob_interpolation}
\HH^{\varepsilon}_p(\mathbb{R}^d)=[\LL^p(\mathbb{R}^d), \W^{m,p}(\mathbb{R}^d)]_{\frac{\varepsilon}{m}},
\end{equation}
where $[E,F]_{\theta}$ denotes the interpolation space between the two Banach spaces $E, F$. For a precise definition and basic properties about interpolation spaces, Theorem 11.6.1 in section 11.6 used to conclude equation \eqref{eq:domain_A} and Corollary 12.3.5 in section 12.3 used for equation \eqref{eq:sob_interpolation},
see \cite{fractional_operators}.
Therefore for $\theta \in (0,1)$:
\[
D( \left(\mathbb{I} - A \right)^{\theta/2}) = \HH^{\theta}_p(\T^3).
\]

\subsection{Technical estimates}

\begin{proof}[Proof of Lemma \ref{lem:compactness_noises}]
We recall that the following Sobolev embedding holds:
\[
\HH^{3+\alpha-6/p}_2(\T^3) \subset \HH^{\alpha}_p(\T^3).
\]
See \cite{hitchhiker} for the proof. Using the above Sobolev embedding result one has:
\[
\begin{aligned}
& \E \left[ \left\|\frac{1}{N} \sum_{i=1}^N \int_0^t(\indicator-A)^{\alpha / 2} e^{(t-s) A} \nabla_x \mathfrak{F}_s^i(x) dB_s^i\right\|_{\LL^p}^q \right] \\
& \quad \leqslant C \cdot \E \left[ \left\|\frac{1}{N} \sum_{i=1}^N \int_0^t(\indicator-A)^{\left(3+\alpha-\frac{6}{p}\right) / 2} e^{(t-s) A} \nabla_x \mathfrak{F}_s^i(x) dB_s^i\right\|_{\LL^2}^q \right].
\end{aligned}
\]
For the above inequality we used that the spatial differential operator $(\indicator - A)^{\beta}$ commutes with the integral $\int_0^t dB_s$. This is a straightforward consequence of Theorem 3.1.2 of \cite{kunita}.
Now let $0<\delta<1$ be such that:
\[
\frac{1}{2} \beta(3+\alpha - 6/p + \delta ) - 1 < 0.
\]
From the Burkholder-Davis-Gundy inequality (see \cite{bdg} for instance), we obtain:
\begin{equation}
\label{eq:estimating_ind.noises}
\begin{aligned}
& \E \left[ \left\|\frac{1}{N} \sum_{i=1}^N \int_0^t(\indicator-A)^{\left(3+\alpha-\frac{6}{p}\right) / 2} e^{(t-s) A}\nabla_x \mathfrak{F}_s^i(x) dB_s^i\right\|_{\LL^2}^q \right] \\
& \leq C_q \mathbb{E}\left[ \left( \frac{1}{N^2} \sum_{i=1}^N \int_0^t\left\|(\indicator-A)^{\left(3+\alpha-\frac{6}{p}\right) / 2} e^{(t-s) A} \nabla_x \mathfrak{F}_s^i(x) \right\|_{\mathbb{L}^2\left(\T^2\right)}^2 d s \right)^{q / 2} \right] \\
& = C_q \mathbb{E}\left[ \left( \frac{1}{N^2} \sum_{i=1}^N \int_0^t\left\|(\indicator -A)^{\frac{1-\delta}{2}} e^{(t-s) A} \right\|_{\LL^2 \to \LL^2}^2 \left\| (\indicator - \Delta)^{\left(3+\alpha-\frac{6}{p} + 1 - 1 + \delta \right) / 2} \mathfrak{F}_s^i(x) \right\|_{\mathbb{L}^2\left(\T^2\right)}^2 d s \right)^{q / 2} \right] \\
& = C_q \mathbb{E}\left[ \left( \frac{1}{N^2} \sum_{i=1}^N \int_0^t \frac{1}{(t-s)^{1-\delta}} \left\| (\indicator - \Delta)^{\left(3+\alpha-\frac{6}{p} + \delta \right) / 2} \left(V^N\left(x-X_s^{i, N}\right)\right) \right\|_{\mathbb{L}^2\left(\T^2\right)}^2 \left| \Phi^{i, N}_s h_0(X_0) \right|^2 ds \right)^{q / 2} \right] \\
\end{aligned}
\end{equation}
Since the following estimate holds:
\[
\left\| (\indicator - \Delta)^{\left(3+\alpha-\frac{6}{p} + \delta \right) / 2} V^N \right\|_{\mathbb{L}^2} \leq C_V N^{\beta(3+\alpha - 6/p + \delta)},
\]
we can deduce from equation \eqref{eq:estimating_ind.noises}:
\begin{equation}
    \begin{aligned}
        & \E \left[ \left\|\frac{1}{N} \sum_{i=1}^N \int_0^t(\mathrm{I}-A)^{\left(3+\alpha-\frac{6}{p}\right) / 2} e^{(t-s) A}\nabla_x \mathfrak{F}_s^i(x) dB_s^i\right\|_{\LL^2}^q \right] \\
        & \leq C_q \mathbb{E}\left[ \left( \frac{1}{N^2} \sum_{i=1}^N \int_0^t \frac{1}{(t-s)^{1-\delta}} \left\| (\indicator - \Delta)^{\left(3+\alpha-\frac{6}{p} + \delta \right) / 2} \left(V^N\left(x-X_s^{i, N}\right)\right) \right\|_{\mathbb{L}^2\left(\T^2\right)}^2 \left| \Phi^{i, N}_s h_0(X_0) \right|^2 ds \right)^{q / 2} \right] \\
        & \leq C_q N^{\frac{q}{2} \beta(3+\alpha - 6/p + \delta ) - q} \cdot \mathbb{E}\left[ \sup_{t \in [0, T]} \left( \sum_{i=1}^N \left| \Phi^{i, N}_s h_0(X_0) \right|^2 ds \right)^{q / 2} \right].
    \end{aligned}
\end{equation}
Therefore we are left with estimating:
\begin{equation}
\begin{aligned}
\mathbb{E}\left[ \sup_{t \in [0, T]} \left(  \sum_{i=1}^N \left| \Phi^{i, N}_t h_0(X_0) \right|^2 ds \right)^{q / 2} \right] = \mathbb{E}\left[ \sup_{t \in [0, T]} \left| \Phi^N \right|_2^{q} \right],
\end{aligned}
\end{equation}
where $\Phi^N_t \in (\R^{3})^N$ is defined as:
\[
(\Phi^N_t)_i := \Phi^{i, N}_th_0(X_0).
\]
From the Burkholder-Davis-Gundy inequality (see \cite{bdg} again)
\begin{equation}
\label{eq:bound_phi_N}
\begin{aligned}
& \mathbb{E}\left[ \sup_{t \in [0, T]} \left| \Phi^N_t \right|_2^{q} \right] \leq \mathbb{E}\left[ \sup_{t \in [0, T]} \left| \text{Id} \right|_2^{q} \right] + \mathbb{E}\left[ \left( \sum\limits_{i=1}^N\int_0^t \left| \left(\nabla \sigma \right) \Phi^{i, N}_s h_0(X_0)\right|^{2} ds \right)^{q/2} \right] \\
& \leq C + \left\| h_0 \right\|_{\LL^{\infty}}^q \left\| \nabla \sigma \right\|_{\LL^{\infty}}^q \mathbb{E}\left[ \left( \sum\limits_{i=1}^N\int_0^t \left| \Phi^{i, N}_s \right|^{2} ds \right)^{q/2} \right] \\
& = C + \left\| h_0 \right\|_{\LL^{\infty}}^q \left\| \nabla \sigma \right\|_{\LL^{\infty}}^q \E \left[ \left(   \int_0^t \left|\Phi^N_s \right|_2^{2} ds \right)^{q/2}
\right] \\
& \leq C + \left\| h_0 \right\|_{\LL^{\infty}}^q \left\| \nabla \sigma \right\|_{\LL^{\infty}}^q \E \left[ \int_0^t \left|\Phi^N_s \right|_2^{q}
 ds \right] \\
& = C + \left\| h_0 \right\|_{\LL^{\infty}}^q \left\| \nabla \sigma \right\|_{\LL^{\infty}}^q \int_0^t \E \left[ \left|\Phi^N_s \right|_2^{q}
\right] ds.
\end{aligned}
\end{equation}
Using the Gronwall's lemma one concludes that:
\begin{equation}
\label{eq:bound_Phi}
\sup\limits_{N \in \N}\mathbb{E}\left[ \sup_{t \in [0, T]} \left| \Phi^N_t \right|_2^{q} \right] < +\infty
\end{equation}
Now, because of the choice of $\beta$, there exists a $\delta > 0$ such that:
\[
N^{\frac{q}{2} \beta(3+\alpha - 6/p + \delta ) - q} \leq 1, \quad \text{ for all } N \in \N.
\]
The proof of the second inequality is similar and we omit it.
\end{proof}

\begin{lem}
\label{lem:banach_spaces_L1}
Let $f:[0, T] \times X \times Y \to \R^d$, and let $p, q > 1$. Then the following inequality holds:
\[
\left\| \int_0^T f(s, x, y) ds\right\|_{\LL^q_y(\LL^p_x)} \leq \int_{0}^T \left\| f(s, x, y) \right\|_{\LL^q_y(\LL^p_x)} ds
\]
\end{lem}
\begin{proof}
It holds for the functions of the form:
\[
f(s, x, y) = \sum\limits_{i=0}^{n +1 } f_{t_i}(x, y) \indicator_{[t_i, t_{i+1}]}(s), \quad \text{ for } 0=t_0<\dots<t_{n+1}=T.
\]
Therefore we conclude by a density argument.
\end{proof}

\begin{proof}[Proof of Lemma \ref{lem:compact_advection_term}]
Let us start with \eqref{eq:mild_term_1}. By Lemma \ref{lem:banach_spaces_L1}:
\[
\begin{aligned}
& \E \left[ \left\| \int_0^t (\indicator - A)^{\frac{\alpha}{2}}e^{(t-s)A} \left( \nabla V^N \star \left( K_{\alpha} \star g_s^{N}  \right) \mu_s^N \right) ds \right\|_{\LL^p}^q \right]^{1/q} \\
&\leq \int_0^t \E \left[ \left\| (\indicator - A)^{\frac{\alpha}{2}}e^{(t-s)A} \left( \nabla V^N \star \left( K_{\alpha} \star g_s^{N}  \right) \mu_s^N \right) \right\|_{\LL^p}^q \right]^{1/q} ds
\end{aligned}
\]
Now, let $0<\delta<1$ such that $\alpha/2 +\delta - 1/2 \leq 0$. Then:
\[
\begin{aligned}
& \int_0^t \E \left[ \left\| (\indicator - A)^{\frac{\alpha}{2}}e^{(t-s)A} \left( \nabla V^N \star \left( K_{\alpha} \star g_s^{N}  \right) \mu_s^N \right) \right\|_{\LL^p}^q \right]^{1/q} ds \\
&= \int_0^t \E \left[ \left\| (\indicator - A)^{1-\delta}e^{(t-s)A}(\indicator - A)^{\frac{\alpha}{2} + \delta - 1} \left( \nabla V^N \star \left( K_{\alpha} \star g_s^{N}  \right) \mu_s^N \right) \right\|_{\LL^p}^q \right]^{1/q} ds \\
& \leq \int_0^t \frac{1}{(t-s)^{1-\delta}}\E \left[ \left\| (\indicator - A)^{\frac{\alpha}{2} + \delta - 1} \left( \nabla V^N \star \left( K_{\alpha} \star g_s^{N}  \right) \mu_s^N \right) \right\|_{\LL^p}^q \right]^{1/q} ds \\
& \leq C \int_0^t \frac{1}{(t-s)^{1-\delta}}\E \left[ \left\| (\indicator - A)^{\frac{\alpha}{2} + \delta - \frac{1}{2}} \left( V^N \star \left( K_{\alpha} \star g_s^{N}  \right) \mu_s^N \right) \right\|_{\LL^p}^q \right]^{1/q} ds \\
& \leq C \int_0^t \frac{1}{(t-s)^{1-\delta}}\E \left[ \left\| V^N \star \left( K_{\alpha} \star g_s^{N}  \right) \mu_s^N\right\|_{\LL^p}^q \right]^{1/q} ds
\end{aligned}
\]
Now one has:
\[
\begin{aligned}
\bigg| V^N \star \left( K_{\alpha} \star g_s^{N}  \right) \mu_s^N (x)\bigg| & = \bigg| \frac{1}{N} \sum\limits_{i=1}^N V^N(x-X_{s}^{i, N}) K_{\alpha} \star g_s^{N}(X^{i, N}_s) \Phi^{i, N}_s h_0(X_0) \bigg| \\
& \leq \left\| K_{\alpha} \star g_s^{N} \right\|_{\LL^{\infty}} \frac{1}{N} \sum\limits_{i=1}^N V^N(x-X_{s}^{i, N}) \bigg|\Phi^{i, N}_s h_0(X_0) \bigg|,
\end{aligned}
\]
and therefore:
\[
\begin{aligned}
& \left(\int_{\T^3}\bigg| V^N \star \left( K_{\alpha} \star g_s^{N}  \right) \mu_s^N (x)\bigg|^p dx \right)^{1/p} = \left( \int_{\T^3}\bigg| \frac{1}{N} \sum\limits_{i=1}^N V^N(x-X_{s}^{i, N}) K_{\alpha} \star g_s^{N}(X^{i, N}_s) \Phi^{i, N}_s h_0(X_0) \bigg|^p dx \right)^{1/p} \\
& \leq \left\| K_{\alpha} \star g_s^{N} \right\|_{\LL^{\infty}} \frac{1}{N} \left( \int_{\T^3} \left( \sum\limits_{i=1}^N V^N(x-X_{s}^{i, N}) \bigg|\Phi^{i, N}_s h_0(X_0) \bigg| \right)^p dx \right)^{1/p} \\
& \leq \left\| K_{\alpha} \star g_s^{N} \right\|_{\LL^{\infty}} \left( \int_{\T^3} \left( \sum\limits_{i=1}^N  \frac{1}{N^2} (\indicator-A)^{\left(3-\frac{6}{p}\right) / 2} V^N(x-X_{s}^{i, N}) \bigg|\Phi^{i, N}_s h_0(X_0) \bigg| \right)^2 dx \right)^{1/2} \\
&= \left\| K_{\alpha} \star g_s^{N} \right\|_{\LL^{\infty}} \frac{1}{N} \left(\sum\limits_{i=1}^N \bigg|\Phi^{i, N}_s h_0(X_0) \bigg|^2 \right)^{1/2}\left\| (\indicator-A)^{\left(3-\frac{6}{p}\right) / 2} V^N \right\|_{\LL^2} \\
& \leq \left\| K_{\alpha} \star g_s^{N} \right\|_{\LL^{\infty}} N^{\beta \left(3-\frac{6}{p}\right) - 1} \left(\sum\limits_{i=1}^N \bigg|\Phi^{i, N}_s h_0(X_0) \bigg|^2 \right)^{1/2}
\end{aligned}
\]
Therefore we conclude:
\[
\begin{aligned}
    &\E \left[ \left\| V^N \star \left( K_{\alpha} \star g_s^{N}  \right) \mu_s^N\right\|_{\LL^p}^q \right]^{1/q} ds \\
    &\leq \E \left[ \left\| K_{\alpha} \star g_s^{N} \right\|_{\LL^{\infty}}^q \right]^{1/q} N^{\beta \left(3-\frac{6}{p}\right) - 1} \E \left[ \left( \sum\limits_{i=1}^N \bigg|\Phi^{i, N}_s h_0(X_0) \bigg|^2 \right)^{q/2} \right]^{1/q}
    \leq C \E \left[ \left\| K_{\alpha} \star g_s^{N} \right\|_{\LL^{\infty}}^q \right]^{1/q}
\end{aligned}
\]
where for the last inequality we used the computation made in \eqref{eq:bound_phi_N}.

Now, for \eqref{eq:mild_term_2} one gets:
\[
\begin{aligned}
    \E \left[ \left\| V^N \star \left( K_{\alpha} \star g_s^{N}  \right) \mu_s^N\right\|_{\LL^p}^q \right]^{1/q} ds \\
    \leq C \E \left[ \left\| \nabla K_{\alpha} \star g_s^{N} \right\|_{\LL^{\infty}}^q \right]^{1/q}
\end{aligned}
\]
We conclude using Lemma \ref{lem:property_of_Kalpha}.
\end{proof}

\begin{lem}
\label{lem:property_of_Kalpha}
Given a $p>1$, there exists a positive constant $C$ that is universal such that for every function $f$:
\begin{equation}
    \begin{aligned}
        \left\| \nabla K_{\alpha} \star f \right\|_{\LL^{\infty}} \leq C \left\| f\right\|_{\LL^{p}},
    \end{aligned}
\end{equation}
and:
\begin{equation}
    \begin{aligned}
        \left\| K_{\alpha} \star f \right\|_{\LL^{\infty}} \leq C \left\| f\right\|_{\LL^{p}},
    \end{aligned}
\end{equation}
\end{lem}

\begin{proof}
Straight consequence of Proposition \ref{th:continuity_of_K_operator}.
\end{proof}

\begin{lem}
\label{lem:-2comma2_case}
For every $q' \geq 2$ there exists a positive constant $C$ such that the following holds:
\[
\E \left[ \left\|\frac{1}{N} \sum\limits_{i = 1}^N V^N(x-X^{i, N}_s) u^{N}(X^{i, N}_s) \Phi^{i, N}_sh_0(X_0^i) \right\|_{\HH^{-1}_2(\T^3)}^{q'}\right] \leq C
\]
\end{lem}

\begin{proof}
Let $\psi \in \HH^{1}_2(\T^3)$ be a test function to compute the norm $\left\|\cdot \right\|_{\HH^{-1}_2(\T^3)}$:
\[
\begin{aligned}
& \left| \int_{\T^3} \frac{1}{N} \sum\limits_{i = 1}^N V^N(x-X^{i, N}_s) u^{N}(X^{i, N}_s) \Phi^{i, N}_sh_0(X_0^i) \psi(x) dx \right| \\
& \leq \int_{\T^3} \frac{1}{N} \sum\limits_{i = 1}^N V^N(x-X^{i, N}_s)  \left|u^{N}(X^{i, N}_s)\right|\left|\Phi^{i, N}_sh_0(X_0^i)\right| \left| \psi(x) \right| dx \\
& \leq \left\| u^N \right\|_{\LL^{\infty}} \frac{1}{N} \sum\limits_{i = 1}^N \int_{\T^3} V^N(x-X^{i, N}_s)  \left| \psi(x) \right| dx \left|\Phi^{i, N}_sh_0(X_0^i)\right| \\
& \leq \left\| u^N \right\|_{\LL^{\infty}} \left\| \psi \right\|_{\LL^{2}} \left\| V^N \right\|_{\LL^{2}} \frac{1}{N} \sum\limits_{i = 1}^N \left|\Phi^{i, N}_sh_0(X_0^i)\right| \\
& \leq  \left\| K_{\alpha} \star g^N_s \right\|_{\LL^{\infty}} \left\| \psi \right\|_{\LL^{2}} \left\| V^N \right\|_{\LL^{2}} \frac{N^{1/2}}{N} \left( \sum\limits_{i = 1}^N \left|\Phi^{i, N}_sh_0(X_0^i)\right|^2 \right)^{1/2} \\
& \leq \left\| K_{\alpha} \star g^N_s \right\|_{\LL^{\infty}} \left\| \psi \right\|_{\LL^{2}}  \frac{N^{1/2+ 3\beta/2}}{N} \left( \sum\limits_{i = 1}^N \left|\Phi^{i, N}_sh_0(X_0^i)\right|^2 \right)^{1/2},
\end{aligned}
\]
which implies:
\[
\left\|\frac{1}{N} \sum\limits_{i = 1}^N V^N(x-X^{i, N}_s) u^{N}(X^{i, N}_s) \Phi^{i, N}_sh_0(X_0^i) \right\|_{\HH^{-1}_2(\T^3)} \leq \left\| K_{\alpha} \star g^N_s \right\|_{\LL^{\infty}} N^{-1/2+ 3\beta/2}\left( \sum\limits_{i = 1}^N \left|\Phi^{i, N}_sh_0(X_0^i)\right|^2 \right)^{1/2}
\]
and therefore:
\[
\begin{aligned}
&\E \left[ \left\|\frac{1}{N} \sum\limits_{i = 1}^N V^N(x-X^{i, N}_s) u^{N}(X^{i, N}_s) \Phi^{i, N}_sh_0(X_0^i) \right\|_{\HH^{-1}_2(\T^3)}^{q'}\right] \\
&\leq N^{-q'/2+ 3q'\beta/2} \E \left[\left\| K_{\alpha} \star g^N_s \right\|_{\LL^{\infty}}^{q'} \left( \sum\limits_{i = 1}^N \left|\Phi^{i, N}_sh_0(X_0^i)\right|^2 \right)^{q'/2} \right] \\
& \leq N^{-q'/2+ 3q'\beta/2} \E \left[\left\| K_{\alpha} \star g^N_s \right\|_{\LL^{\infty}}^{2q'}\right]^{1/2} \E \left[\left( \sum\limits_{i = 1}^N \left|\Phi^{i, N}_sh_0(X_0^i)\right|^2 \right)^{2q'/2}\right]^{1/2} \\
& \leq N^{-q'/2+ 3q'\beta/2} \E \left[\left\| g^N_s \right\|_{\LL^{p}}^{2q'}\right]^{1/2} \E \left[\left( \sum\limits_{i = 1}^N \left|\Phi^{i, N}_sh_0(X_0^i)\right|^2 \right)^{2q'/2}\right]^{1/2} \leq C
\end{aligned}
\]
where for the last inequality we used the inequality \eqref{eq:bound_Phi}, the Theorem \ref{thm:compact_1}, and the fact that $3\beta/2 - 1/2 < 0$.
\end{proof}

\begin{lem}
\label{lem:-2comma2_case_noise}
For every $q' \geq 2$ there exists a positive constant $C$ such that the following holds:
\[
\E \left[ \left\|\frac{1}{N} \sum\limits_{i = 1}^N \int_s^t \nabla V^N(x-X^{i, N}_s) \sigma (X^{i, N}_s)\Phi^{i, N}_sh_0(X_0^i) dB_s^i\right\|_{\HH^{-2}_2(\T^3)}^{q'}\right] \leq C(t-s)^{q'/2}
\]
\end{lem}
\begin{proof}
From the Burkholder-Davis-Gundy inequality (see \cite{bdg}):
\[
\begin{aligned}
& \E \left[ \left\|\frac{1}{N} \sum\limits_{i = 1}^N \int_s^t \nabla V^N(x-X^{i, N}_r) \sigma (X^{i, N}_r)\Phi^{i, N}_r h_0(X_0^i) dB_s^i\right\|_{\HH^{-2}_2(\T^3)}^{q'}\right] \\
& \leq C \mathbb{E}\left[\frac{1}{N^2} \sum_{i=1}^N \int_s^t\left\|\nabla V^N\left(x-X_r^{i, N}\right)\sigma(X^{i, N}_r)\Phi^{i, N}_r h_0(X_0^i)\right\|_{\HH^{-2}_2(\T^3)}^{2} d r\right]^{q^{\prime} / 2}.
\end{aligned}
\]
Therefore we are left with proving that there exists a positive constant $C$ such that:
\[
\frac{1}{N^{q'}} \sum_{i=1}^N \mathbb{E}\left[  \left\|V^N\left(x-X_r^{i, N}\right)\sigma(X^{i, N}_r)\Phi^{i, N}_r h_0(X_0^i)\right\|_{\HH^{-1}_2(\T^3)}^{2} \right] \leq C.
\]
The proof of this is exactly the same as the one for Lemma \ref{lem:-2comma2_case}.
\end{proof}

\begin{proof}[Proof of Theorem \ref{thm:compact_2}]
We have:
\[
\begin{aligned}
& g^{N}_{t}(x) - g^{N}_{s}(x) - \int_{s}^{t} A g^N_r dr \\
&- \int_{s}^{t} \left\langle \mu^{N}_{r}, \left( \sigma \cdot \nabla \sigma \right) \cdot \nabla V^{N}(x - \cdot) \right\rangle dr - \int_{s}^{t} \left\langle \mu^{N}_{r}, \left( \left( \nabla \sigma \right) \cdot \nabla \left( \nabla \sigma \right) \right) \cdot \nabla V^{N}(x - \cdot) \right\rangle dr\\
&- \int_{s}^{t} \left\langle \mu^{N}_{r}, u^{N}_{r} \cdot \nabla V^{N}(x - \cdot) \right\rangle dr - \int_{s}^{t} \left\langle \mu^{N}_{r}, V^{N}(x - \cdot) \cdot \nabla u^{N}_{r} \right\rangle dr \\
& - \int_{s}^{t}\left\langle \mu^{N}_{r}, \sigma \cdot \nabla V^{N}(x - \cdot) \right\rangle dW_r - \int_{s}^{t}\left\langle \mu^{N}_{r},  V^{N}(x - \cdot) \cdot \nabla \sigma \right\rangle dW_r\\
& = \frac{\sqrt{2\nu}}{N} \sum\limits_{i = 1}^{N}\int_{s}^{t} \nabla V^{N}(x-X_{r}^{i, N}) \Phi_{r}^{i, N} h_{0}(X_{0}^i) dB_{r}^{i},
\end{aligned}
\]
from which:
\[
\begin{aligned}
& \E \left[ \left\|g^{N}_{t}(x) - g^{N}_{s}(x) \right\|_{\HH^{-2}_2(\T^3)}^{q'}\right] \leq (t-s)^{q'-1} \int_{s}^{t} \E \left[ \left\|A g^N_r dr \right\|_{\HH^{-2}_2(\T^3)}^{q'}\right] \\
&+ (t-s)^{q'-1} \int_{s}^{t} \E \left[ \left\| \left\langle \mu^{N}_{r}, \left( \sigma \cdot \nabla \sigma \right) \cdot \nabla V^{N}(x - \cdot) \right\rangle \right\|_{\HH^{-2}_2(\T^3)}^{q'}\right] dr \\
& + (t-s)^{q'-1} \int_{s}^{t} \E \left[ \left\| \left\langle \mu^{N}_{r}, \left( \left( \nabla \sigma \right) \cdot \nabla \left( \nabla \sigma \right) \right) \cdot \nabla V^{N}(x - \cdot) \right\rangle \right\|_{\HH^{-2}_2(\T^3)}^{q'}\right] dr\\
&+ (t-s)^{q'-1} \int_{s}^{t} \E \left[ \left\| \left\langle \mu^{N}_{r}, u^{N}_{r} \cdot \nabla V^{N}(x - \cdot) \right\rangle \right\|_{\HH^{-2}_2(\T^3)}^{q'}\right] dr \\
& + (t-s)^{q'-1} \int_{s}^{t} \E \left[ 
 \left\| \left\langle \mu^{N}_{r}, V^{N}(x - \cdot) \cdot \nabla u^{N}_{r} \right\rangle \right\|_{\HH^{-2}_2(\T^3)}^{q'}\right] dr \\
& + \E \left[ 
 \left\| \int_{s}^{t} \left\langle \mu^{N}_{r}, \sigma \cdot \nabla V^{N}(x - \cdot) \right\rangle dW_r \right\|_{\HH^{-2}_2(\T^3)}^{q'}\right] + \E \left[ 
 \left\| \int_{s}^{t}\left\langle \mu^{N}_{r},  V^{N}(x - \cdot) \cdot \nabla \sigma \right\rangle dW_r \right\|_{\HH^{-2}_2(\T^3)}^{q'}\right] \\
& = \E \left[ 
 \left\| \frac{\sqrt{2\nu}}{N} \sum\limits_{i = 1}^{N}\int_{s}^{t} \nabla V^{N}(x-X_{r}^{i, N}) \Phi_{r}^{i, N} h_{0}(X_{0}^i) dB_{r}^{i} \right\|_{\HH^{-2}_2(\T^3)}^{q'}\right].
\end{aligned}
\]
Now, from Lemma \ref{lem:-2comma2_case} we are able to say that the first five lines above are bounded by:
\[
C (t-s)^{q' - 1} (t-s) =C (t-s)^{q'}.
\]
From Lemma \ref{lem:-2comma2_case_noise} we can infer that all the terms in the above concerning the stochastic integrals are bounded by:
\[
(t-s)^{q'/2}.
\]
This is enough to conclude, because:
\[
\mathbb{E}\left[\int_0^T \int_0^T \frac{\left\|g_t^N-g_s^N\right\|_{\HH^{-2}_2(\T^3)}^{q^{\prime}}}{|t-s|^{1+q^{\prime} \gamma}} d s d t\right] \leq C \int_0^T \int_0^T \frac{|t-s|^{q'/2}}{|t-s|^{1+q^{\prime} \gamma}} d s d t < \infty.
\]
\end{proof}

\subsection{A representation formula for the Navier-Stokes-alpha model}
\label{repform}

The following formula has been inspired by
Remark 2.5 in \cite{constantinStochasticLagrangianRepresentation2008}. In \cite{drivas_holm}, the authors prove a more general connection with the deterministic Navier Stokes equations that also implies the Theorem \ref{thm:representation_formula}.

Just as above, \((B_{t})_{t \geq 0}\) will be a \(3\)-dimensional Brownian
motion, independent from \((W_{t})_{t \geq 0}\), that is a
\(1\)-dimensional Brownian motion.

Consider the following system and assume it has a solution:
\begin{equation}
\begin{cases}
dX = udt + \sqrt{2\nu} dB + \sigma(X) \circ dW, \\
A_t=X_t^{-1}, \\
v = \mathbb{E}\left[ \left( \nabla^TA \right) u_{0}(A) | W\right], \\
u = \left( \indicator - \alpha^2 \Delta \right)^{-1} v
\end{cases}
\end{equation}
and let \(\omega = \nabla \times v\) be the vorticity. In the above system, we mean that $A_t$ is the inverse of $X_t : \T^3 \to \T^3$ as a diffeomorphism. See Theorem 3.1.2 of \cite{kunita}.

Now consider:
\[
\tilde{\omega} = ((\nabla X) \omega_{0}) (A).\]
We will show in this section that $\tilde{\omega}$ satisfies the following:
\[
d \tilde{\omega} = \left[ - \left( u \cdot \nabla \right) \tilde{\omega} + \Delta \tilde{\omega} + \left( \nabla^{T} u\right) \tilde{\omega} \right] dt - \sqrt{2\nu} \nabla \tilde{\omega} dB - \left(\nabla \tilde{\omega} \right) \sigma  \circ dW + \left( \nabla \sigma \right) \tilde{\omega} \circ dW.
\]

The equation satisfied by \(X\) can also be written with only It\^{o} integrals
as: \[
\begin{cases}
dX = \left(u(X) + \frac{1}{2}(\nabla\sigma) \sigma (X)\right)dt + \sqrt{2\nu} dB + \sigma(X)dW, \\
A=X^{-1}, \\
v = \mathbb{E}\left[ \left( \nabla^TA \right) u_{0}(A) | W\right], \\
u = \left( \indicator - \alpha^2 \Delta \right)^{-1} v.
\end{cases}
\]

\begin{lem}
There exists a process \(\Lambda\) with bounded variation such
that: \[
A_{t} = \Lambda_{t} - \sqrt{2\nu} \int_{0}^{t} \nabla A_{s}dB_{s} - \int_{0}^{t}(\nabla A_{s}) \sigma dW. \]
\end{lem}

\begin{proof}
Let us apply the generalized It\^{o} formula (see \cite{kunita}). One has:
\[\begin{aligned} 0 & = \int_{t^{\prime}}^t A\left(X_s, d s\right)+\left.\int_{t^{\prime}}^t \nabla A\right|_{X_s, s} d X_s+\left.\frac{1}{2} \int_{t^{\prime}}^t \partial_{i j}^2 A\right|_{X_s, s} d\left\langle X^{i}, X^{j}\right\rangle_s+ \\ & +\left\langle\int_{t^{\prime}}^t \partial_i A\left(X_s, d s\right), X_{t}^{i}-X_{t^{\prime}}^{i}\right\rangle \\ 
&= \int_{t^{\prime}}^t A\left(X_s, d s\right)+\int_{t^{\prime}}^t\left[\left.\nabla A\right|_{X_{s, s}}\left(u + \frac{1}{2}(\nabla\sigma) \sigma \right)+\left.\Phi A\right|_{X_s, s}\right] d s+\left.\sqrt{2 \nu} \int_{t^{\prime}}^t \nabla A\right|_{X_s, s} d B_{s}+ \\
& + \left. \int_{t^{\prime}}^{t} \nabla A\right|_{X_s, s}  \sigma dW + \left\langle\int_{t^{\prime}}^t \partial_i A\left(X_s, d s\right), X_{t}^{i}-X_{t^{\prime}}^{i}\right\rangle.
\end{aligned}
\]
where
\(\Phi = \nu \Delta + \frac{1}{2}\sum\limits_{i,j}\sigma^{i} \sigma^{j}\partial_{i}\partial_{j}\).
The lemma follows because \(X\) is an homeomorphism.
\end{proof}

Now let us prove the following:

\begin{lem}
\label{lem:equation_for_A}
The process \(A\) satisfies the stochastic partial
differential equation \[
d A_{t}+\left(u + \frac{1}{2}(\nabla\sigma) \sigma \right) \cdot \nabla A d t-\Phi A d t+\sqrt{2 \nu} \nabla A d B_{t} + (\nabla A)\sigma dW_{t}= 0.\]
\end{lem}

\begin{proof}
We can just use the computation
made in the previous lemma and notice that: \[
\left\langle\int_{t^{\prime}}^{t} \partial_{i} A\left(X_{s}, d s\right), X_{t}^{i}-X_{t^{\prime}}^{i}\right\rangle = - 2 \int_{t^{\prime}}^{t} \Phi A ds.
\]
\end{proof}

\begin{lem}
Let \(v = \theta(A)\) where \(\theta\) can also be itself a
semimartingale. Then:
\[
d v_{t}+ \left(u+ \frac{1}{2}(\nabla\sigma) \sigma \right) \cdot \nabla v_{t} d t-\Phi v_{t} d t+\sqrt{2 \nu} \nabla v_t d B_{t}+ \nabla v_{t} \sigma dW_{t}=\left.d \vartheta\right|_{A_t}.\]
\end{lem}

\begin{proof}
This is an application of the
generalised It\^{o} formula and of previous Lemma \ref{lem:equation_for_A}.
In fact: 
\[
\begin{aligned}
d v_t= & \left. d\vartheta\right|_{A_t}+\left.\nabla \vartheta\right|_{A_t} d A_t+\left.\frac{1}{2} \partial_{i j}^{2} \vartheta\right|_{A_t} d\left\langle A^{i}, A^{j}\right\rangle_t \\ = & \bigg[\left. d\vartheta\right|_{A_t}-\left.\nabla \vartheta\right|_{A_t}\left(\nabla A_t\right) \left(u_{t} + \frac{1}{2}(\nabla\sigma) \sigma \right)dt+\left. \nabla \vartheta\right|_{A_t} \Phi A_{t} dt \\
&+ \left. \nu \partial_{i j}^{2} \vartheta\right|_{A_t} \partial_{k} A_{t}^{i} \partial_{k} A_{t}^{j} + \partial_{i j}^{2} \vartheta \left[ \nabla A \cdot \sigma \right]^{i} \left[ \nabla A \cdot \sigma \right]^{j}\bigg] d t- \\ & \quad-\left.\sqrt{2 \nu} \nabla \vartheta\right|_{A_{t}} \nabla A_{t} d B_{t} - \left.\nabla \vartheta\right|_{A_{t}} \nabla A_{t} \sigma dW_s\\ = & {\left[\left. d\vartheta\right|_{A_t}-\left(u_{t} + \frac{1}{2}(\nabla\sigma) \sigma \right)  \cdot \nabla v_t+\Phi v_t\right] d t-\sqrt{2 \nu} \nabla v_{t} d B_{t} - \left.\nabla \vartheta\right|_{A_{t}} \nabla A_{t} \sigma dW_{s}}.
\end{aligned}
\]
\end{proof}

\begin{lem}
\label{lem:omega_tilde}
If one defines:
\[
\tilde{\omega} = ((\nabla X) \omega_{0}) (A),\]
then it satisfies:
\[
d \tilde{\omega} = \left[ - \left(u + \frac{1}{2}(\nabla\sigma) \sigma \right) \cdot \nabla  \tilde{\omega} + \Phi \tilde{\omega} + \left( \nabla^{T} u\right) \tilde{\omega} \right] dt - \sqrt{2\nu} \nabla \tilde{\omega} dB - \left(\nabla \tilde{\omega} \right) \sigma dW + \left( \nabla \sigma \right) \tilde{\omega} \circ dW.
\]
\end{lem}

\begin{proof}
Notice first that \(\nabla X\) is differentiable in time. We set
\(\vartheta=(\nabla X) \omega_0, \tilde{\omega}=\vartheta \circ A\) and
apply the previous Lemma to obtain

\[
\begin{aligned}
& d \tilde{\omega}+\left(u + \frac{1}{2}(\nabla\sigma) \sigma\right)\cdot \nabla \tilde{\omega} d t-\nu \Phi \tilde{\omega} d t+\sqrt{2 \nu} \nabla \tilde{\omega} d B + \nabla \tilde{\omega} \sigma dW =\left.\nabla d X\right|_{A_t} \omega_0(A) \\
& =\left.(\nabla u)(\nabla X)\right|_A \omega_0(A) d t + \left.(\nabla \sigma)(\nabla X)\right|_A \omega_0(A) \circ dW\\
& =(\nabla u) \tilde{\omega} d t + (\nabla \sigma)\tilde{\omega} \circ dW
\end{aligned}
\]
\end{proof}

\begin{thm}
\label{thm:representation_formula}
The following function:
\[
\omega=\mathbb{E}\left[ \tilde{\omega} | W \right],
\]
satisfies the following SPDE:
\[
d \omega + \left(u + \frac{1}{2}(\nabla\sigma) \sigma\right) \cdot \nabla \omega dt -\nu \Phi \omega d t + \left( \nabla \omega \right) \sigma dW = (\nabla u) \omega dt + (\nabla \sigma)\omega \circ dW.
\]
\end{thm}

\begin{proof}
To prove the theorem, it is enough to use Lemma \ref{lem:omega_tilde}, integrate the equation and pass to the conditional expectation with respect to the filtration generated by $W$.

The only things to be careful about are:
\begin{itemize}
    \item since \(u = \left(\indicator - \alpha^2 \Delta\right)^{-1}\mathbb{E}\left[ \left( \nabla^T A \right) u_{0}(A)|W\right]\), it is \(W\)
measurable, and therefore is goes out from the conditional expectation.
    \item The following holds:
\[\mathbb{E}\left[\int_{0}^{t} X_{s}dW_{s} \bigg| W\right] =\int_{0}^{t}\mathbb{E}\left[ X_{s}|W\right]dW_{s}\]
\end{itemize}
\end{proof}

\begin{remark}
Let us consider a test function $\phi : \T^3 \to \R^3$ and $\omega$ solution to equation \eqref{eq:NS_alpha_model}. Then:
\begin{equation}
\label{eq:eur}
\begin{aligned}
    \int_{\T^3} \omega(y) \phi(y) dy & = \E \left[ \int_{\T^3} (\nabla X) \omega_0 (X^{-1}(y)) \phi(y) dy \bigg|W\right] = \E \left[ \int_{\T^3} (\nabla X) \omega_0 (x) \phi(X(x)) |\det (\nabla X)|(x)dx \bigg|W\right] \\
    & = \E \left[ \int_{\T^3} (\nabla X) \omega_0 (x) \phi(X(x)) dx \bigg|W\right] = \int_{\T^3} \E \left[ (\nabla X) \omega_0 (x) \phi(X(x)) \bigg|W\right] dx
\end{aligned}
\end{equation}

Let $(\eta^i)_{i =1}^N$ be a lattice in $\T^3$ such that the volume of a single cube is $1/N$. Heuristically, denoting $X^{i, N} = X(\eta^i)$, and $\Phi^{i, N} =  (\nabla X)(\eta^i)$ one has:
\[
\begin{aligned}
\int_{\T^3} \E \left[ (\nabla X) \omega_0 (x) \phi(X(x)) \bigg|W\right] dx & \approx \sum\limits_{i} \E \left[ (\nabla X)(\eta^i) \left[ \omega_0(\eta^i) \right] \phi(X(\eta^i)) \bigg|W \right] \cdot \frac{1}{N}\\
& = \frac{1}{N}\sum\limits_{i} \E \left[ \Phi^{i, N} \left[ \omega_0(\eta^i) \right] \phi(X^{i, N}) \bigg|W \right].
\end{aligned}
\]
We can see here the similarity with the definition of $\mu^N$ in equation \eqref{eq:action_mu_test}.
\end{remark}

\bibliographystyle{abbrvnat}
\bibliography{bibliography}

\begin{thebibliography}{49}
\providecommand{\natexlab}[1]{#1}
\providecommand{\url}[1]{\texttt{#1}}
\expandafter\ifx\csname urlstyle\endcsname\relax
  \providecommand{\doi}[1]{doi: #1}\else
  \providecommand{\doi}{doi: \begingroup \urlstyle{rm}\Url}\fi

\bibitem[Albeverio and Belopolskaya(2010)]{albeverio_belopolskaya}
S.~Albeverio and Y.~Belopolskaya.
\newblock Generalized solutions of the {C}auchy problem for the {N}avier-{S}tokes system and diffusion processes.
\newblock \emph{Cubo}, 12\penalty0 (2):\penalty0 77--96, 2010.
\newblock ISSN 0716-7776,0719-0646.
\newblock \doi{10.4067/s0719-06462010000200006}.
\newblock URL \url{https://doi.org/10.4067/s0719-06462010000200006}.

\bibitem[Anderson and Greengard(1985)]{anderson_greengard}
C.~Anderson and C.~Greengard.
\newblock On vortex methods.
\newblock \emph{SIAM J. Numer. Anal.}, 22\penalty0 (3):\penalty0 413--440, 1985.
\newblock ISSN 0036-1429.
\newblock \doi{10.1137/0722025}.
\newblock URL \url{https://doi.org/10.1137/0722025}.

\bibitem[Aubin(1998)]{MR1636569}
T.~Aubin.
\newblock \emph{Some nonlinear problems in {R}iemannian geometry}.
\newblock Springer Monographs in Mathematics. Springer-Verlag, Berlin, 1998.
\newblock ISBN 3-540-60752-8.
\newblock \doi{10.1007/978-3-662-13006-3}.
\newblock URL \url{https://doi.org/10.1007/978-3-662-13006-3}.

\bibitem[Beale(1986)]{beale_grid_free}
J.~T. Beale.
\newblock A convergent {$3$}-{D} vortex method with grid-free stretching.
\newblock \emph{Math. Comp.}, 46\penalty0 (174):\penalty0 401--424, S15--S20, 1986.
\newblock ISSN 0025-5718,1088-6842.
\newblock \doi{10.2307/2007984}.
\newblock URL \url{https://doi.org/10.2307/2007984}.

\bibitem[Beale and Majda(1982{\natexlab{a}})]{beale_madja_I}
J.~T. Beale and A.~Majda.
\newblock Vortex methods. i: Convergence in three dimensions.
\newblock \emph{Mathematics of Computation}, 39\penalty0 (159):\penalty0 1--27, 1982{\natexlab{a}}.
\newblock ISSN 00255718, 10886842.

\bibitem[Beale and Majda(1982{\natexlab{b}})]{beale_majda_II}
J.~T. Beale and A.~Majda.
\newblock Vortex methods. ii: Higher order accuracy in two and three dimensions.
\newblock \emph{Mathematics of Computation}, 39\penalty0 (159):\penalty0 29--52, 1982{\natexlab{b}}.
\newblock ISSN 00255718, 10886842.

\bibitem[Billingsley(1999)]{skorohod}
P.~Billingsley.
\newblock \emph{Convergence of probability measures}.
\newblock Wiley Series in Probability and Statistics: Probability and Statistics. John Wiley \& Sons, Inc., New York, second edition, 1999.
\newblock ISBN 0-471-19745-9.
\newblock \doi{10.1002/9780470316962}.
\newblock URL \url{https://doi.org/10.1002/9780470316962}.
\newblock A Wiley-Interscience Publication.

\bibitem[Busnello et~al.(2005)Busnello, Flandoli, and Romito]{brusnello_flandoli_romito}
B.~Busnello, F.~Flandoli, and M.~Romito.
\newblock A probabilistic representation for the vorticity of a three-dimensional viscous fluid and for general systems of parabolic equations.
\newblock \emph{Proc. Edinb. Math. Soc. (2)}, 48\penalty0 (2):\penalty0 295--336, 2005.
\newblock ISSN 0013-0915,1464-3839.
\newblock \doi{10.1017/S0013091503000506}.
\newblock URL \url{https://doi.org/10.1017/S0013091503000506}.

\bibitem[Chen et~al.(1998)Chen, Foias, Holm, Olson, Titi, and Wynne]{MR1745983}
S.~Chen, C.~Foias, D.~D. Holm, E.~Olson, E.~S. Titi, and S.~Wynne.
\newblock Camassa-{H}olm equations as a closure model for turbulent channel and pipe flow.
\newblock \emph{Phys. Rev. Lett.}, 81\penalty0 (24):\penalty0 5338--5341, 1998.
\newblock ISSN 0031-9007,1079-7114.
\newblock \doi{10.1103/PhysRevLett.81.5338}.
\newblock URL \url{https://doi.org/10.1103/PhysRevLett.81.5338}.

\bibitem[Chen et~al.(1999{\natexlab{a}})Chen, Foias, Holm, Olson, Titi, and Wynne]{MR1719962}
S.~Chen, C.~Foias, D.~D. Holm, E.~Olson, E.~S. Titi, and S.~Wynne.
\newblock A connection between the {C}amassa-{H}olm equations and turbulent flows in channels and pipes.
\newblock volume~11, pages 2343--2353. 1999{\natexlab{a}}.
\newblock \doi{10.1063/1.870096}.
\newblock URL \url{https://doi.org/10.1063/1.870096}.
\newblock The International Conference on Turbulence (Los Alamos, NM, 1998).

\bibitem[Chen et~al.(1999{\natexlab{b}})Chen, Foias, Holm, Olson, Titi, and Wynne]{MR1721139}
S.~Chen, C.~Foias, D.~D. Holm, E.~Olson, E.~S. Titi, and S.~Wynne.
\newblock The {C}amassa-{H}olm equations and turbulence.
\newblock volume 133, pages 49--65. 1999{\natexlab{b}}.
\newblock \doi{10.1016/S0167-2789(99)00098-6}.
\newblock URL \url{https://doi.org/10.1016/S0167-2789(99)00098-6}.
\newblock Predictability: quantifying uncertainty in models of complex phenomena (Los Alamos, NM, 1998).

\bibitem[Chorin(2013)]{Chorin1994VorticityAT}
A.~Chorin.
\newblock \emph{Vorticity and Turbulence}.
\newblock Applied Mathematical Sciences. Springer New York, 2013.
\newblock ISBN 9781441987280.

\bibitem[Chorin(1973)]{chorin_main}
A.~J. Chorin.
\newblock Numerical study of slightly viscous flow.
\newblock \emph{J. Fluid Mech.}, 57\penalty0 (4):\penalty0 785--796, 1973.
\newblock ISSN 0022-1120,1469-7645.
\newblock \doi{10.1017/S0022112073002016}.
\newblock URL \url{https://doi.org/10.1017/S0022112073002016}.

\bibitem[Chorin and Bernard(1973)]{chorin1973discretization}
A.~J. Chorin and P.~S. Bernard.
\newblock Discretization of a vortex sheet, with an example of roll-up.
\newblock \emph{Journal of Computational Physics}, 13\penalty0 (3):\penalty0 423--429, 1973.

\bibitem[Constantin and Iyer(2008)]{constantinStochasticLagrangianRepresentation2008}
P.~Constantin and G.~Iyer.
\newblock A stochastic {L}agrangian representation of the three-dimensional incompressible {N}avier-{S}tokes equations.
\newblock \emph{Comm. Pure Appl. Math.}, 61\penalty0 (3):\penalty0 330--345, 2008.
\newblock ISSN 0010-3640,1097-0312.
\newblock \doi{10.1002/cpa.20192}.
\newblock URL \url{https://doi.org/10.1002/cpa.20192}.

\bibitem[Cottet and Koumoutsakos(2000)]{cottet_koumoutsakos}
G.-H. Cottet and P.~D. Koumoutsakos.
\newblock \emph{Vortex methods}.
\newblock Cambridge University Press, Cambridge, 2000.
\newblock ISBN 0-521-62186-0.
\newblock \doi{10.1017/CBO9780511526442}.
\newblock URL \url{https://doi.org/10.1017/CBO9780511526442}.
\newblock Theory and practice.

\bibitem[Cottet et~al.(1991)Cottet, Goodman, and Hou]{cottet_goodman_hou}
G.-H. Cottet, J.~Goodman, and T.~Y. Hou.
\newblock Convergence of the grid-free point vortex method for the three-dimensional {E}uler equations.
\newblock \emph{SIAM J. Numer. Anal.}, 28\penalty0 (2):\penalty0 291--307, 1991.
\newblock ISSN 0036-1429.
\newblock \doi{10.1137/0728016}.
\newblock URL \url{https://doi.org/10.1137/0728016}.

\bibitem[Da~Prato and Zabczyk(2014)]{Da_Prato_Zabczyk_2014}
G.~Da~Prato and J.~Zabczyk.
\newblock \emph{Stochastic Equations in Infinite Dimensions}.
\newblock Encyclopedia of Mathematics and its Applications. Cambridge University Press, 2 edition, 2014.

\bibitem[Di~Nezza et~al.(2012)Di~Nezza, Palatucci, and Valdinoci]{hitchhiker}
E.~Di~Nezza, G.~Palatucci, and E.~Valdinoci.
\newblock Hitchhiker's guide to the fractional {S}obolev spaces.
\newblock \emph{Bull. Sci. Math.}, 136\penalty0 (5):\penalty0 521--573, 2012.
\newblock ISSN 0007-4497,1952-4773.
\newblock \doi{10.1016/j.bulsci.2011.12.004}.
\newblock URL \url{https://doi.org/10.1016/j.bulsci.2011.12.004}.

\bibitem[Drivas and Holm(2020)]{drivas_holm}
T.~D. Drivas and D.~D. Holm.
\newblock Circulation and energy theorem preserving stochastic fluids.
\newblock \emph{Proc. Roy. Soc. Edinburgh Sect. A}, 150\penalty0 (6):\penalty0 2776--2814, 2020.
\newblock ISSN 0308-2105,1473-7124.
\newblock \doi{10.1017/prm.2019.43}.
\newblock URL \url{https://doi.org/10.1017/prm.2019.43}.

\bibitem[Esposito and Pulvirenti(1989)]{esposito_pulvirenti}
R.~Esposito and M.~Pulvirenti.
\newblock Three-dimensional stochastic vortex flows.
\newblock \emph{Math. Methods Appl. Sci.}, 11\penalty0 (4):\penalty0 431--445, 1989.
\newblock ISSN 0170-4214,1099-1476.
\newblock \doi{10.1002/mma.1670110402}.
\newblock URL \url{https://doi.org/10.1002/mma.1670110402}.

\bibitem[Flandoli et~al.(2020)Flandoli, Olivera, and Simon]{paper}
F.~Flandoli, C.~Olivera, and M.~Simon.
\newblock Uniform approximation of 2 dimensional {N}avier-{S}tokes equation by stochastic interacting particle systems.
\newblock \emph{SIAM J. Math. Anal.}, 52\penalty0 (6):\penalty0 5339--5362, 2020.
\newblock ISSN 0036-1410.
\newblock \doi{10.1137/20M1328993}.

\bibitem[Foias et~al.(2001)Foias, Holm, and Titi]{foias}
C.~Foias, D.~D. Holm, and E.~S. Titi.
\newblock The navier–stokes-alpha model of fluid turbulence.
\newblock \emph{Physica D: Nonlinear Phenomena}, 152-153:\penalty0 505--519, 2001.
\newblock ISSN 0167-2789.
\newblock \doi{https://doi.org/10.1016/S0167-2789(01)00191-9}.
\newblock URL \url{https://www.sciencedirect.com/science/article/pii/S0167278901001919}.
\newblock Advances in Nonlinear Mathematics and Science: A Special Issue to Honor Vladimir Zakharov.

\bibitem[Foias et~al.(2002)Foias, Holm, and Titi]{MR1878243}
C.~Foias, D.~D. Holm, and E.~S. Titi.
\newblock The three dimensional viscous {C}amassa-{H}olm equations, and their relation to the {N}avier-{S}tokes equations and turbulence theory.
\newblock \emph{J. Dynam. Differential Equations}, 14\penalty0 (1):\penalty0 1--35, 2002.
\newblock ISSN 1040-7294,1572-9222.
\newblock \doi{10.1023/A:1012984210582}.
\newblock URL \url{https://doi.org/10.1023/A:1012984210582}.

\bibitem[Fontbona(2006)]{fontbona}
J.~Fontbona.
\newblock A probabilistic interpretation and stochastic particle approximations of the 3-dimensional navier-stokes equations.
\newblock \emph{Probability Theory and Related Fields}, 136, 09 2006.
\newblock \doi{10.1007/s00440-005-0477-9}.

\bibitem[Giovagnini and Crisan(2024)]{giovagnini2024uniformpointvortexapproximation}
F.~Giovagnini and D.~Crisan.
\newblock A uniform point vortex approximation for the solution of the two-dimensional navier stokes equation with transport noise, 2024.
\newblock URL \url{https://arxiv.org/abs/2410.23163}.

\bibitem[Glatt-Holtz and Ziane(2009)]{MR2502705}
N.~Glatt-Holtz and M.~Ziane.
\newblock Strong pathwise solutions of the stochastic {N}avier-{S}tokes system.
\newblock \emph{Adv. Differential Equations}, 14\penalty0 (5-6):\penalty0 567--600, 2009.
\newblock ISSN 1079-9389.

\bibitem[Goodman(1987)]{goodman1987convergence}
J.~Goodman.
\newblock Convergence of the random vortex method.
\newblock \emph{Communications on Pure and Applied Mathematics}, 40\penalty0 (2):\penalty0 189--220, 1987.

\bibitem[Gy\"{o}ngy and Krylov(2022)]{gyongy}
I.~Gy\"{o}ngy and N.~V. Krylov.
\newblock Existence of strong solutions for {I}t\^{o}'s stochastic equations via approximations: revisited.
\newblock \emph{Stoch. Partial Differ. Equ. Anal. Comput.}, 10\penalty0 (3):\penalty0 693--719, 2022.
\newblock ISSN 2194-0401,2194-041X.
\newblock \doi{10.1007/s40072-022-00273-7}.

\bibitem[{Kazantsev}(1968)]{kazantsev}
A.~P. {Kazantsev}.
\newblock {Enhancement of a Magnetic Field by a Conducting Fluid}.
\newblock \emph{Soviet Journal of Experimental and Theoretical Physics}, 26:\penalty0 1031, May 1968.

\bibitem[Kirchhoff(1876)]{kirchhoff1876}
G.~Kirchhoff.
\newblock \emph{Vorlesungen über mathematische Physik: Mechanik}.
\newblock Teubner, Leipzig, 1876.
\newblock Reprinted editions are also available.

\bibitem[Knio and Ghoniem(1990)]{knio_ghoniem}
O.~M. Knio and A.~F. Ghoniem.
\newblock Numerical study of a three-dimensional vortex method.
\newblock \emph{J. Comput. Phys.}, 86\penalty0 (1):\penalty0 75--106, 1990.
\newblock ISSN 0021-9991,1090-2716.
\newblock \doi{10.1016/0021-9991(90)90092-F}.
\newblock URL \url{https://doi.org/10.1016/0021-9991(90)90092-F}.

\bibitem[{Kolmogorov}(1941)]{kolmogorov}
A.~{Kolmogorov}.
\newblock {The Local Structure of Turbulence in Incompressible Viscous Fluid for Very Large Reynolds' Numbers}.
\newblock \emph{Akademiia Nauk SSSR Doklady}, 30:\penalty0 301--305, Jan. 1941.

\bibitem[Kunita(1997)]{kunita}
H.~Kunita.
\newblock \emph{Stochastic Flows and Stochastic Differential Equations}.
\newblock Cambridge Studies in Advanced Mathematics. Cambridge University Press, 1997.

\bibitem[Long(1988)]{Long1988}
D.-G. Long.
\newblock Convergence of the random vortex method in two dimensions.
\newblock \emph{Journal of the American Mathematical Society}, 1\penalty0 (4):\penalty0 779--804, 1988.
\newblock ISSN 08940347, 10886834.

\bibitem[Majda and Bertozzi(2002)]{madja_bertozzi}
A.~J. Majda and A.~L. Bertozzi.
\newblock \emph{Vorticity and incompressible flow}, volume~27 of \emph{Cambridge Texts in Applied Mathematics}.
\newblock Cambridge University Press, Cambridge, 2002.

\bibitem[Marchioro and Pulvirenti(1994)]{marchioro_pulvirenti}
C.~Marchioro and M.~Pulvirenti.
\newblock \emph{Mathematical theory of incompressible nonviscous fluids}, volume~96 of \emph{Applied Mathematical Sciences}.
\newblock Springer-Verlag, New York, 1994.
\newblock ISBN 0-387-94044-8.
\newblock \doi{10.1007/978-1-4612-4284-0}.

\bibitem[Mart\'{\i}nez~Carracedo and Sanz~Alix(2001)]{fractional_operators}
C.~Mart\'{\i}nez~Carracedo and M.~Sanz~Alix.
\newblock \emph{The theory of fractional powers of operators}, volume 187 of \emph{North-Holland Mathematics Studies}.
\newblock North-Holland Publishing Co., Amsterdam, 2001.
\newblock ISBN 0-444-88797-0.

\bibitem[Nualart(2006)]{nualart}
D.~Nualart.
\newblock \emph{The {M}alliavin calculus and related topics}.
\newblock Probability and its Applications (New York). Springer-Verlag, Berlin, second edition, 2006.

\bibitem[Olivera(2021)]{olivera_representation}
C.~Olivera.
\newblock Probabilistic representation for mild solution of the {N}avier-{S}tokes equations.
\newblock \emph{Math. Res. Lett.}, 28\penalty0 (2):\penalty0 563--573, 2021.
\newblock ISSN 1073-2780,1945-001X.
\newblock \doi{10.4310/MRL.2021.v28.n2.a8}.
\newblock URL \url{https://doi.org/10.4310/MRL.2021.v28.n2.a8}.

\bibitem[Pazy(1983)]{pazy}
A.~Pazy.
\newblock \emph{Semigroups of linear operators and applications to partial differential equations}, volume~44 of \emph{Applied Mathematical Sciences}.
\newblock Springer-Verlag, New York, 1983.
\newblock ISBN 0-387-90845-5.
\newblock \doi{10.1007/978-1-4612-5561-1}.

\bibitem[Qian et~al.(2022)Qian, S\"uli, and Zhang]{qian}
Z.~Qian, E.~S\"uli, and Y.~Zhang.
\newblock Random vortex dynamics via functional stochastic differential equations.
\newblock \emph{Proc. A.}, 478\penalty0 (2266):\penalty0 Paper No. 20220030, 23, 2022.
\newblock ISSN 1364-5021,1471-2946.
\newblock \doi{10.1098/rspa.2022.0030}.
\newblock URL \url{https://doi.org/10.1098/rspa.2022.0030}.

\bibitem[Rosenhead(1932)]{rosenhead}
L.~Rosenhead.
\newblock The point vortex approximation of a vortex sheet.
\newblock \emph{Proc. R. Soc. Lond. A}, 134:\penalty0 170--192, 1932.

\bibitem[Sikora and Wright(2001)]{sikora}
A.~Sikora and J.~Wright.
\newblock Imaginary powers of {L}aplace operators.
\newblock \emph{Proc. Amer. Math. Soc.}, 129\penalty0 (6):\penalty0 1745--1754, 2001.
\newblock ISSN 0002-9939,1088-6826.
\newblock \doi{10.1090/S0002-9939-00-05754-3}.
\newblock URL \url{https://doi.org/10.1090/S0002-9939-00-05754-3}.

\bibitem[Simon(1986)]{simon}
J.~Simon.
\newblock Compact sets in the space $l^p(0,t; b)$.
\newblock \emph{Annali di Matematica Pura ed Applicata}, 146\penalty0 (1):\penalty0 65--96, 1986.
\newblock ISSN 1618-1891.
\newblock \doi{10.1007/BF01762360}.

\bibitem[van Neerven et~al.(2007)van Neerven, Veraar, and Weis]{bdg}
J.~M. A.~M. van Neerven, M.~C. Veraar, and L.~Weis.
\newblock Stochastic integration in {UMD} {B}anach spaces.
\newblock \emph{Ann. Probab.}, 35\penalty0 (4):\penalty0 1438--1478, 2007.
\newblock ISSN 0091-1798,2168-894X.
\newblock \doi{10.1214/009117906000001006}.
\newblock URL \url{https://doi.org/10.1214/009117906000001006}.

\bibitem[Vincenzi(2002)]{kraichnan_kazantsev}
D.~Vincenzi.
\newblock The {K}raichnan-{K}azantsev dynamo.
\newblock \emph{J. Statist. Phys.}, 106\penalty0 (5-6):\penalty0 1073--1091, 2002.
\newblock ISSN 0022-4715,1572-9613.
\newblock \doi{10.1023/A:1014089820881}.
\newblock URL \url{https://doi.org/10.1023/A:1014089820881}.

\bibitem[von Helmholtz(1858)]{helmholtz1858}
H.~von Helmholtz.
\newblock Über integrale der hydrodynamischen gleichungen, welche den wirbelbewegungen entsprechen.
\newblock \emph{Journal für die reine und angewandte Mathematik}, 55:\penalty0 25--55, 1858.
\newblock Translated as "On Integrals of the Hydrodynamic Equations which Correspond to Vortex Motions".

\bibitem[Zhang(2016)]{zhang}
X.~Zhang.
\newblock Stochastic differential equations with {S}obolev diffusion and singular drift and applications.
\newblock \emph{Ann. Appl. Probab.}, 26\penalty0 (5):\penalty0 2697--2732, 2016.
\newblock ISSN 1050-5164,2168-8737.
\newblock \doi{10.1214/15-AAP1159}.
\newblock URL \url{https://doi.org/10.1214/15-AAP1159}.

\end{thebibliography}
\end{document}